\documentclass{article}[11pt]

\usepackage{amsmath}   
\usepackage{amssymb}
\usepackage{accents}
\usepackage{amsthm}
\usepackage{tikz-cd} 
\usepackage{xcolor}
\usepackage{hyperref}

\newtheorem{Th}{Theorem}[section] 
\newtheorem{Prop}{Proposition}[section]   
\newtheorem{Lem}{Lemma}[section]   
\newtheorem{Coro}{Corollary}[section]   
\newtheorem{Def}{Definition}  
\newtheorem{Rem}{Remark}[section]

\newcommand{\R}{\mathbb{R}}
\newcommand{\Z}{\mathbb{Z}}
\newcommand{\C}{\mathbb{C}}

\newcommand{\bbS}{\mathbb{S}}
\newcommand{\s}{{\rm S}}
\newcommand{\x}{\langle x\rangle}
\newcommand{\y}{\langle y\rangle}
\newcommand{\z}{\langle z\rangle}
\newcommand{\LL}{\mathcal{L}}
\newcommand{\A}{{\mathcal A}}

\newcommand{\K}{\mathcal{K}}

\newcommand{\U}{{\mathcal U}}

\newcommand{\Sz}{\mathcal{S}}
\newcommand{\p}{{\rm p}}

\newcommand{\tr}{\mathop{\rm tr}}
\newcommand{\re}{\mathop{\rm Re}}

\newcommand{\Div}{\mathop{\rm div}\nolimits}

\newcommand{\tz}{{\widetilde z}}

\begin{document}

\title{Spatial asymptotic expansions in the Navier-Stokes equation}   
 
\author{R. McOwen and P. Topalov} 

\maketitle

\begin{abstract}  
We prove that the Navier-Stokes equation for a viscous incompressible fluid in $\R^d$ is locally well-posed in spaces 
of functions allowing spatial asymptotic expansions with log terms as $|x|\to\infty$ of any a priori given order. 
The solution depends analytically on the initial data and time so that for any $0<\vartheta<\pi/2$ it can be 
holomorphically extended in time to a conic sector in $\C$ with angle $2\vartheta$ at zero.
We discuss the approximation of solutions by their asymptotic parts.
\end{abstract}   


\section{Introduction}\label{sec:introduction}
The evolution of the velocity field of an incompressible viscous fluid in $\R^d$ is described by the Navier-Stokes equation
\begin{equation}
\left\{
\begin{array}{l}\label{eq:NS}
u_t+u\cdot\nabla u=\nu \Delta u-\nabla\p,\quad\Div u=0,\\
u|_{t=0}=u_0,
\end{array}
\right.
\end{equation}
where $u(x,t)$ is the velocity field on $\R^d$ at time $t>0$, $u\cdot\nabla :=\sum_{j=1}^d u_j \frac{\partial}{\partial x_j}$ is the derivative in 
the direction of $u$, $\p(x,t)$ is the scalar pressure, and the constant $\nu>0$ is the viscosity of the fluid.  We are interested whether
for a given integer $N\ge 0$ and an initial data $u_0$ with an asymptotic expansion  
\begin{equation}\label{eq:asymptotic_expansion_schwartz}
u_0(x)=a_0(\theta)+\frac{a_1(\theta)}{r}+\cdots+\frac{a_N(\theta)}{r^N}+o\left(\frac{1}{r^N}\right)
\quad\hbox{as $|x|\to\infty$},
\end{equation}
where $r:=|x|$, $\theta:=x/|x|\in\s^{d-1}\equiv\big\{\,x\in\R^d\big|\,|x|=1\big\}$, 
and $a_k : \s^{d-1}\to\R^d$ ($0\le k\le N$) are bounded functions,
the solution of \eqref{eq:NS} admits an asymptotic expansion of the same form for each $0\le t\le T$ with $T>0$.
However, for the Euler equation ($\nu=0$), the asymptotic expansion of the solution for $t>0$ may involve log terms in the following form 
(cf. \cite{McOwenTopalov3})
\begin{equation}\label{eq:asymptotic_expansion}
a_0(\theta)+\frac{a^0_1(\theta)+a_1^1(\theta)\log r}{r}+\cdots+\frac{a^0_N(\theta)+\cdots+a_N^N(\theta)(\log r)^N}{r^N}+
o\left(\frac{1}{r^N}\right)\ \hbox{as $|x|\to\infty$}.
\end{equation}
In fact, it is shown in \cite{McOwenTopalov3} that the Euler flow is well-posed in a Banach space of (real) vector fields on $\R^d$ with 
asymptotics \eqref{eq:asymptotic_expansion}, which we denote by $\A^{m,p}_{N;0}$ where $m\ge 0$ is a regularity exponent
and $1<p<\infty$. In this paper, we shall obtain results on the spatial asymptotics of solutions for the Navier-Stokes equation. 

Note that spatial asymptotics of the form \eqref{eq:asymptotic_expansion_schwartz} appear {\em naturally} in the solutions of
the Euler and the Navier-Stokes equation due to the presence of the (non-local) pressure term in \eqref{eq:NS}. 
We refer the reader to the short discussion of related works at the end of this introduction. 
At this point we only note that there is a generic set of divergence free initial data $u_0$ in the Schwartz space $\Sz$ such that the solution $u$ of
the Euler equation has an asymptotic expansion of the form \eqref{eq:asymptotic_expansion_schwartz} with asymptotic terms of 
{\em all} orders starting with $a_{d+1}(\theta)/r^{d+1}$ appearing and do {\em not} vanishing generically in time 
(see \cite[Theorem 1.2, Corollary 1.1]{McOwenTopalov4}). Similar results can be obtained for the Navier-Stokes equation 
(see the related work section below).
In this way it is natural to ask what will happen if we start with initial data $u_0$ that is not necessarily short-range but has spatial asymptotic
expansion of the form \eqref{eq:asymptotic_expansion_schwartz}. 
(We refer to Example 1 in \cite[Appendix]{McOwenTopalov3} for the necessity of including log terms in the asymptotic expansions.)
As mentioned above, we are also interested in how the asymptotics and the remainder evolve in time or how they both depend on time 
and the initial data. 

\medskip

We shall give the definition of $\A^{m,p}_{N;0}$ below, but for now let us state the results that we prove in this paper.
We let $B^{m,p}_{N;0}(\rho)$ denote the open ball in $\A^{m,p}_{N;0}$ of radius $\rho>0$ and let $\accentset{\,\,\,\circ}{\A}^{m,p}_{N;0}$ 
denote the divergence free vector fields in $\A^{m,p}_{N;0}$. For the notion of solution we refer to Remark \ref{rem:definition_solution} in 
Section \ref{sec:navier-stokes}.

\begin{Th}\label{th:NS}
Assume that $m>2+\frac{d}{p}$, $N\in\Z_{\ge 0}$, and $1<p<\infty$. For any $\rho>0$ there exists $T\equiv T(\rho)>0$ so that
for any divergence free vector field $u_0\in B^{m,p}_{N;0}(\rho)$ there exists a unique solution 
$u\in C\big([0,T],\accentset{\,\,\,\circ}{\A}^{m,p}_{N;0}\big)\cap C^1\big([0,T],\accentset{\,\,\,\circ}{\A}^{m-2,p}_{N;0}\big)$ of 
the Navier-Stokes equation \eqref{eq:NS} that depends Lipschitz continuously on the initial data in the sense that the data-to-solution map,
\[
u_0\mapsto u,\quad B^{m,p}_{N;0}(\rho)\cap\accentset{\,\,\,\circ}{\A}^{m,p}_{N;0} \to 
C\big([0,T],\accentset{\,\,\,\circ}{\A}^{m,p}_{N;0}\big)\cap C^1\big([0,T],\accentset{\,\,\,\circ}{\A}^{m-2,p}_{N;0}\big),
\]
is Lipschitz continuous.
\end{Th}

In fact, in Section \ref{sec:navier-stokes} we prove a stronger result: For a given $T>0$ and an angle $\vartheta\in(0,\pi/2)$ 
consider the conic sectors in $\C$,
\begin{equation}\label{eq:conic_sets}
\bbS_\vartheta:=\big\{z\in\C\,\big|\,z\ne 0, |\arg z|<\vartheta\big\}\quad\text{\rm and}\quad
\bbS_{\vartheta,T}:=\big\{z\in\bbS_{\vartheta}\,\big|\,|z|<T\big\}.
\end{equation}
In Theorem \ref{th:NS-complex} we show that for any $\rho>0$ and $\vartheta\in(0,\pi/2)$ the time of existence 
$T>0$ in Theorem \ref{th:NS} above can be chosen so that the unique solution 
$u\in C\big([0,T],\accentset{\,\,\,\circ}{\A}^{m,p}_{N;0}\big)\cap C^1\big([0,T],\accentset{\,\,\,\circ}{\A}^{m-2,p}_{N;0}\big)$ 
of the Navier-Stokes equation extends to a bounded holomorpic map\footnote{$X_\C$ denotes the complexification of a given real Banach space $X$.}
\[
u : \bbS_{\vartheta,T}\to\accentset{\,\,\,\circ}{\A}^{m,p}_{N;0,\C}.
\]
The holomorphic map $u\in C_b\big(\bbS_{\vartheta,T},\accentset{\,\,\,\circ}{\A}^{m,p}_{N;0,\C}\big)$ depends Lipschitz continuously on 
the initial data $u_0\in B^{m,p}_{N;0}(\rho)\cap\accentset{\,\,\,\circ}{\A}^{m,p}_{N;0}$.
(For a given metric space $D$ and a Banach space $Y$, $C_b(D,Y)$ denotes the Banach space of bounded continuous maps $D\to Y$ 
equipped with the $\sup$-norm.)
In particular, for any $u_0\in\accentset{\,\,\,\circ}{\A}^{m,p}_{N;0}$ there exists a maximal time of existence $T_\infty>0$ and
a unique solution
$C\big([0,T_\infty),\accentset{\,\,\,\circ}{\A}^{m,p}_{N;0}\big)\cap C^1\big([0,T_\infty),\accentset{\,\,\,\circ}{\A}^{m-2,p}_{N;0}\big)$ 
of the Navier-Stokes equation \eqref{eq:NS} that is real-analytic on $(0,T_\infty)$ (see Corollary \ref{coro:NS-global}).
Moreover, by Cauchy's estimate (see e.g. \cite[Lemma A.2]{PK}), for any given $l\ge 1$, 
$u\in C^l\big((0,T_\infty),\accentset{\,\,\,\circ}{\A}^{m,p}_{N;0}\big)$
depends Lipschitz continuously on the initial data $u_0\in\accentset{\,\,\,\circ}\A^{m,p}_{N;0}$ in the sense that for any compact interval 
$K\subseteq(0,T_\infty)$, $u\in C^l\big(K,\accentset{\,\,\,\circ}{\A}^{m,p}_{N;0}\big)$ depends Lipschitz continuously 
on $\accentset{\,\,\,\circ}{\A}^{m,p}_{N;0}$. In fact, the solution $u$ depends analytically on the initial data
$u_0\in\accentset{\,\,\,\circ}{\A}^{m,p}_{N;0}$ -- see Remark \ref{rem:analytic_dependence} for the details.

The proofs are based on the properties of the heat flow established in Section \ref{sec:A-spaces}, and are independent of 
the proof in the case of the Euler equation. We first prove that the Laplace operator, when considered as an unbounded operator in the asymptotic space 
$\A^{m,p}_{N;0}$, is sectorial with domain ${\widetilde\A}^{m+2,p}_{N;0}$ (see \eqref{def:D(Lap|A)} and Theorem \ref{th:sectorial_A}). 
As a consequence, we obtain that the heat flow generates an analytic semigroup on the asymptotic space $\A^{m,p}_{N;0}$ with 
angle $\vartheta=\pi/2$ (Theorem \ref{th:SonA}). This allows us to prove Corollary \ref{coro:S(t)-estimatesA} and Corollary \ref{coro:S(t)-estimatesA*} 
that contain important estimates on the heat flow used in the prove of Theorem \ref{th:NS-complex}. 
It is worth pointing out that the domain ${\widetilde\A}^{m+2,p}_{N;0}$ of the Laplace operator in $\A^{m,p}_{N;0}$ is characterized explicitly in 
terms of the scale of asymptotic spaces (see \eqref{def:D(Lap|A)}). The results in Section \ref{sec:A-spaces} 
are based on analogous results on weighted Sobolev spaces proved in Section \ref{sec:W-spaces} (see Theorem \ref{th:sectorial_W},
Theorem \ref{th:SonW}, Corollary \ref{coro:S(t)-estimatesW} and Corollary \ref{coro:S(t)-estimatesW*}).

Let us now define the asymptotic spaces. In order to control the remainder in \eqref{eq:asymptotic_expansion} we will assume
that it belongs to a weighted Sobolev space $W^{m,p}_{\gamma_N}$ whose elements are real-valued functions on $\R^d$ with asymptotic
$o(1/r^N)$ as $r\to\infty$. Let us define the space $W^{m,p}_{\gamma_N}$.
We fix $p\in (1,\infty)$, let $C_c^\infty(\R^d)$ denote the space of smooth functions with compact support in $\R^d$, let $\Z_{\ge 0}$ denote 
the nonnegative integers, and set $\x:=\sqrt{1+|x|^2}$. 
For $m\in\Z_{\ge 0}$ and $\delta\in\R$, define the weighted Sobolev space $W_\delta^{m,p}\equiv W_\delta^{m,p}(\R^d)$ to be 
the closure of $C_c^\infty(\R^d)$ in the norm 
\begin{equation}\label{W-norm}
\|f\|_{W_\delta^{m,p}}=\sum_{|\alpha|\leq m}\|\langle x\rangle^{\delta+|\alpha |}\partial^\alpha f \|_{L^p}
\end{equation}
where $\alpha\in\Z_{\ge 0}^d$ is a multi-index, $\partial^\alpha\equiv\partial_1^{\alpha_1}\cdots\partial_d^{\alpha_d}$, and
$\partial_k\equiv\partial/\partial x_k$ is the (weak) partial derivative in the direction of $x_k$.
We also denote $W^{0,p}_\delta$ by $L^p_\delta$. 
It was shown in \cite{Bartnik} and \cite[Lemma 2.2]{McOwenTopalov2} that $f\in W^{m,p}_\delta(\R^d)$ for $m>d/p$
implies that $f\in C^k(\R^d)$ for $0\leq k<m-d/p$ with
\begin{subequations}
\begin{equation}\label{eq:W-estimate}
\sup_{x\in\R^d}\x^{\delta+\frac{d}{p}+|\alpha|} |\partial^\alpha f(x)| \leq C\,\|f\|_{W^{m,p}_\delta}
\,\,\, \hbox{for}\ |\alpha|<m-d/p
\end{equation}
and in fact 
\begin{equation}\label{eq:W-estimate*}
|x|^{\delta+\frac{d}{p}+|\alpha|}|\partial^\alpha f(x)|\to 0 \ \text{\rm as}\ |x|\to\infty
\,\,\, \hbox{for}\ |\alpha|<m-d/p.
\end{equation}
\end{subequations}
Next, we introduce
\begin{equation}\label{def:gamma_N}
\gamma_N:=N+\gamma_0\ \hbox{where $\gamma_0$ is fixed and chosen so that}\ 0\le\gamma_0+\frac{d}{p}<1\, .
\end{equation}
Let us now proceed to the definition of the asymptotic space.
Choose a cut-off function $\chi\in C^\infty_c(\R)$ such that $\chi(r)=0$ for $0\le r\le 1$ and $\chi(r)=1$ for $r\ge 2$. 
For integer $0\le n\le N$ and an integer $\ell$ satisfying $\ell\geq -n$ define 
$\A^{m,p}_{n,N;\ell}\equiv\A^{m,p}_{n,N;\ell}(\R^d)$ to be the space of real-valued functions of the form 
\begin{equation}\label{AWlog-expansion1} 
v(x)=a(x)+f(x),
\end{equation}
where
\begin{equation}\label{AWlog-expansion2} 
a(x)=\chi(r)\left(\frac{a^0_n+\cdots a_n^{n+\ell}(\log r)^{n+\ell}}{r^n}+\cdots + 
\frac{a^0_{N}+\cdots+a_N^{N+\ell}(\log r)^{N+\ell}}{r^{N}}\right)
\end{equation}
with $a^j_k\in H^{m+1+N-k,p}(\s^{d-1})$ is the {\em asymptotic part of $v$} and 
\begin{equation}\label{AW-expansion3} 
f\in W^{m,p}_{\gamma_N} 
\end{equation}
is the {\em remainder of $v$} (cf. \cite[Appendix B]{McOwenTopalov2}, \cite[Appendix C]{McOwenTopalov4} for details).
Note that by \eqref{eq:W-estimate*}, if $f\in W^{m,p}_{\gamma_N}$ with $m>d/p$ then $f=o(1/r^N)$ as $r\to\infty$.
The function space $\A^{m,p}_{n,N;\ell}$ is a Banach space under the norm
\begin{equation}\label{def:AW;-norm}
\|v\|_{{\mathcal A}_{n,N;\ell}^{m,p}}=\sum_{k=n}^N \sum_{j=0}^{k+\ell} \|a^j_k\|_{H^{m+1+N-k,p}}+\|f\|_{W_{\gamma_N}^{m,p}}.
\end{equation}
For $n=0$ we write $\A^{m,p}_{N;\ell}$ instead of $\A^{m,p}_{0,N;\ell}$. We are particularly interested in the case $\ell=0$, which was used 
in our results on Euler's equation (see \cite{McOwenTopalov3}).
We will also need the {\em complex} versions of the spaces $W^{m,p}_\delta$ and $\A^{m,p}_{n,N;\ell}$ which can be defined  in 
a straightforward manner by allowing the functions in the definitions above to take complex values. The norms \eqref{W-norm} and 
\eqref{def:AW;-norm} are then readily extendable to the complex case. Finally, note that for simplicity of notation we will use (whenever possible)
the symbols $\A^{m,p}_{n,N;\ell}$, $W^{m,p}_\delta$, etc., independently of whether we are considering single-valued functions, 
vector fields, or tensor fields whose elements belong to the considered space of functions.
Note that in contrast to \cite{McOwenTopalov2}, in the present paper we will allow $m\ge 0$ in the definition of the remainder \eqref{AW-expansion3}.
(This will not cause problems since $a\notin W^{m,p}_{\gamma_N}$, and hence, the decomposition \eqref{AWlog-expansion1} of 
$v$ into an asymptotic and a reminder part is unique.)

\medskip

Let us now draw conclusions on the dynamics of the solutions of the Navier-Stokes equation.
Choose $m>2+\frac{d}{p}$, $N\ge\Z_{\ge 0}$, and $1<p<\infty$, and take a divergence free initial data 
$u_0\in\accentset{\,\,\,\circ}{\A}^{m,p}_{N;0}$ with weight $\gamma_N$ chosen as in \eqref{def:gamma_N} 
with $\gamma_0+\frac{d}{p}>0$. In particular, we can choose $u_0$ of the form \eqref{eq:asymptotic_expansion_schwartz} or, 
more generally, of the form \eqref{eq:asymptotic_expansion}.
Then, by Theorem \ref{th:NS}, there exist $T>0$ and a unique solution 
$u\in C\big([0,T],\accentset{\,\,\,\circ}{\A}^{m,p}_{N;0}\big)\cap C^1\big([0,T],\accentset{\,\,\,\circ}{\A}^{m-2,p}_{N;0}\big)$ 
of the Navier-Stokes equation. In view of \eqref{eq:W-estimate}, we obtain estimates on the remainder in the asymptotic expansion of $u$:
there exists a constant $C\equiv C_T>0$ such that 
for any $t\in[0,T]$ and for any $|x|\ge 2$ we have
\begin{equation}\label{eq:approximation}
\left|\partial^\alpha\Big(u(x,t)-\sum_{0\le k\le N}\frac{a^0_k(\theta,t)+\cdots+a_k^k(\theta,t)(\log r)^k}{r^k}\Big)\right|\le
\frac{C}{|x|^{N+\varepsilon+|\alpha|}},\quad 0\le |\alpha|<m-\frac{d}{p},
\end{equation}
where $\varepsilon\equiv\gamma_0+\frac{d}{p}>0$.
Since the solution depends continuously on the initial data, the constant $C>0$ can be chosen so that \eqref{eq:approximation} holds 
locally uniformly on the initial data $u_0\in\accentset{\,\,\,\circ}{\A}^{m,p}_{N;0}$.
Therefore, the solution $u(x,t)$ of the Navier-Stokes equation is approximated uniformly on $\R^d$, together with its spatial derivatives 
of order less than $m-\frac{d}{p}$, by its asymptotic part
\begin{equation}\label{eq:asymptotic_part}
a_0(\theta)+\sum_{1\le k\le N}\frac{a^0_k(\theta,t)+\cdots+a_k^k(\theta,t)(\log r)^k}{r^k}
\end{equation}
with error of order $O\big(1/|x|^{N+\varepsilon+|\alpha|}\big)$ uniformly in $t\in[0,T]$ and locally uniformly on the initial data.
This implies that {\em the behavior of the solution $u(x,t)$ (and its derivatives) in the complement of a centered at zero closed ball in $\R^d$ 
is determined by the ansatz \eqref{eq:asymptotic_part} up to an error of a higher order of magnitude}. 
Note that the leading asymptotic term $a_0(\theta)$ of the solution $u(x,t)$ in \eqref{eq:asymptotic_part} is 
independent of time; this follows easily by comparing the leading asymptotic in the Navier-Stokes equation.
Finally, note that one can apply Theorem \ref{th:NS} to initial data $u_0$ in the weighted Sobolev space 
$W^{m,p}_\delta$ with $\delta+\frac{d}{p}\ge 0$. This implies that for any $0<t\le T$ the solution $u(x,t)$ has an asymptoic
expansion of order $N=[\delta+\frac{d}{p}]$ as $|x|\to\infty$ where $[\cdot]$ denotes the integer part of a real number. 
This, and the arguments in \cite{McOwenTopalov4}, can then be used to show that the asymptotic expansion of $u(x,t)$ does not involve 
log terms and that Theorem 1.2 (and Theorem 1.1) in \cite{McOwenTopalov4} continues to hold for the Navier-Stokes equation.
As in \cite{McOwenTopalov4}, one can show that nontrivial asymptotic terms of {\em all} orders starting with $a_{d+1}(\theta)/r^{d+1}$ do appear; this is true for generic initial data $u_0\in W^{m,p}_\delta$.

\medskip
 
\noindent{\em Related work:} There are many important works related to the solutions of the Navier-Stokes equation on 
$\R^d$ in various function spaces; we refer the reader to the monographs \cite{Kato2,MB} and the references therein.
Let us now discuss works concerned with the spatial behavior of solutions at infinity. It was noted in \cite{DobrShaf} that the solutions of 
the Euler and the Navier-Stokes equations in $\R^3$ with initial data in the Schwartz space $\Sz$ have a decay rate of order $O(1/r^4)$
in an arbitrarily small time. The existence of a full asymptotic expansion in the case of Euler equation for $d=2$ with initial data in the Schwartz
space (and, more generally, in weighted Sobolev spaces) is established in \cite[Corollary 1.2 and 1.3]{SultanTopalov}, and 
for initial vorticity with compact support in \cite[Theorem 1.2]{Bran1}.
The general case of $\R^d$ with $d\ge 2$ is considered in \cite[Theorem 1.2, Corollary 1.1]{McOwenTopalov4}
(cf. also \cite[Theorem 1.1]{McOwenTopalov3}). 
The existence of an asymptotic expansion for the solutions of the Navier-Stokes equation for $d=2$ with initial vorticity in a weighted $L^1$-space was
proven in \cite[Theorem 1.1]{Bran1}. 
The case of $\R^d$ with $d\ge 2$ and initial data in the Schwartz space is considered in \cite[Theorem 2.1]{KR}, where the asymptotic expansion
is written in terms of the inverse Fourier transform of given functions. To our best knowledge, there are no prior works concerned with 
the existence of (full) spatial asymptotic expansions of the solutions of the Navier-Stokes equation with non-decaying initial data. 
The dependence of solutions on the initial data has also not been previously discussed. The asymptotic spaces $\A^{m,p}_{N;\ell}$ are 
introduced in \cite{McOwenTopalov1,McOwenTopalov2} as a tool for studying spatial asymptotic expansions of the solutions of PDEs. 
The local well-posedness of the Euler equation in the asymptotic space $\A^{m,p}_{N;0}$ is proven in \cite{McOwenTopalov3,McOwenTopalov4}. 
Applications to the heat flow are obtained in \cite{McOwenTopalov5}; however, there we use slightly different remainder spaces. 
We also mention the work \cite{KS} where the behavior of stationary solutions of the Navier-Stokes equation is studied.
Finally, note that asymptotic type spaces with a single asymptotic term of order $O(1/r)$ were introduced and studied in dimension two:
we refer to the affine space $E_m$ in \cite[Definition 1.3.3]{Chemin}, as well as to the radial-energy decomposition in \cite[Definition 3.1]{MB}.

\section{The heat semigroup on weighted Sobolev spaces}\label{sec:W-spaces}
Recall that the heat equation on $(0,\infty)\times\R^d$,
\begin{equation}\label{eq:heat}
\left\{
\begin{array}{l}
u_t=\nu\Delta u,\\
u|_{t=0}=u_0,
\end{array}
\right.
\end{equation}
where $u_0$ belongs to the Banach space $C_b$ of uniformly bounded continuous vector fields on $\R^d$, 
is solved by $u(t)=S_\nu(t)u_0$, $t\ge 0$, where
\begin{equation}\label{eq:Gaussian}
\big(S_\nu(t)u_0\big)(x):=\frac{1}{(4\pi\nu t)^{d/2}} \int_{\R^d} e^{-\frac{|x-y|^2}{4\nu t}} u_0(y)\,dy,\quad x\in\R^d,
\end{equation}
and the solution satisfies $u\in C^\infty((0,\infty)\times\R^d)\cap C([0,\infty)\times\R^d)$ -- see e.g. \cite[Ch. 7]{Shubin}. 
More generally, \eqref{eq:Gaussian} extends to a semigroup  $\{S_\nu(t)\}_{t\ge 0}$ of bounded linear maps on 
the space of tempered distributions $\Sz'\equiv\Sz(\R^d)$ -- see Lemma \ref{lem:heat_flow_in_S'} and 
Proposition \ref{prop:heat_equation_in_S'} in Appendix \ref{sec:aux-results}. 
We want to consider $S_\nu(t)$ as a semigroup on the weighted Sobolev spaces $W_\delta^{m,p}$. 
Since $W_\delta^{m,p}$ is a subspace of $\Sz'$, we know that for any given $t\ge 0$ the map $S_\nu(t)$ is 
defined on $W_\delta^{m,p}$ but it is not immediately clear that $W_\delta^{m,p}$ is invariant under 
$S_\nu(t)$ or how its operator norm depends on $t\ge 0$. These are addressed in Theorem \ref{th:SonW} and
Corollary \ref{coro:S(t)-estimatesW} below. We set $S(t):=S_\nu(t)\big|_{\nu=1}$.
For simplicity of notation we will denote in this section the real spaces and their complexifications by the same symbols.

\medskip

For integer $m\ge 0$ and $\delta\in\R$ consider the Banach space
\begin{equation}\label{def:D(Lap|W)}
{\widetilde W}^{m+2,p}_\delta:=\big\{f\in \Sz'\,\big|\,\partial^\alpha f\in W^{m,p}_\delta, |\alpha|\leq 2\big\}
\end{equation}
equipped with the norm $\|f\|_{{\widetilde W}^{m+2,p}_\delta}:=\sum_{|\alpha|\le 2}\|\partial^\alpha f\|_{W^{m,p}_{\delta}}$.
Note that ${\widetilde W}^{m+2,p}_\delta$ is a dense subspace in $W^{m,p}_\delta$ and by \eqref{def:D(Lap|W)} we have
\begin{equation}\label{eq:inclusions_W}
W^{m+2,p}_\delta\subseteq{\widetilde W}^{m+2,p}_\delta\subseteq W^{m,p}_\delta
\end{equation}
with bounded inclusion maps. Moreover, by Lemma \ref{lem:domains} in Appendix \ref{sec:aux-results}, we have that 
${\widetilde W}^{m+2,p}_\delta=W^{m,p}_{\delta}\cap W^{m+1,p}_{\delta-1}\cap W^{m+2,p}_{\delta-2}$.
For any $\omega>0$ and any $0<\epsilon<\pi$ consider the cone 
\[
\Sigma_{\omega,\epsilon}:=\big\{\lambda\in\C\,\big|\,\lambda\ne\omega, 
|\!\arg(\lambda-\omega)|<\pi-\epsilon\big\}.
\]
For Banach space $X$ and $Y$ denote by $\LL(X,Y)$ the Banach space of bounded linear maps $X\to Y$.
In the case when $Y\equiv X$ we set $\LL(X)\equiv\LL(X,X)$.
One has the following theorem.

\begin{Th}\label{th:sectorial_W}
Assume that $1<p<\infty$, $\delta\in\R$ and that $m\ge 0$ is an integer. Then, we have
\begin{itemize}
\item[(i)] For any $\lambda\in\C\setminus(-\infty,0]$ the map 
\begin{equation*}
(\lambda-\Delta) : {\widetilde W}^{m+2,p}_\delta\to W^{m,p}_{\delta}
\end{equation*}
is an isomorphism of Banach spaces.
In particular, the (distributional) Laplacian $\Delta$, when considered as an unbounded operator $\Delta|_{W^{m,p}_{\delta}}$ in 
$W^{m,p}_{\delta}$ with domain $D(\Delta|_{W^{m,p}_{\delta}})={\widetilde W}^{m+2,p}_\delta$, is closed and 
its spectrum is contained in $(-\infty,0]$. 

\item[(ii)] For any $\omega>0$ and for any $0<\epsilon<\pi$ there exists a 
positive constant $C\equiv C_{\omega,\epsilon}>0$ such that
\[
\big\|\partial^\alpha(\lambda-\Delta)^{-1}\big\|_{\LL(W^{m,p}_{\delta})}\le C/|\lambda-\omega|^{1-\frac{|\alpha|}{2}}
\]
for any $\lambda\in\Sigma_{\omega,\epsilon}$ and any multi-index $\alpha\in\Z_{\ge 0}^d$ with $|\alpha|\le 2$.
\end{itemize}
\end{Th}

\begin{Rem}\label{rem:spectrum_W}
It follows from \cite{McOwen1} that $\lambda=0$ belongs to the spectrum of $\Delta|_{W^{m,p}_{\delta}}$. More specifically, since 
the domain ${\widetilde W}^{m+2,p}_\delta$ of $\Delta|_{W^{m,p}_{\delta}}$ is a dense proper subspace in $W^{m+2,p}_{\delta-2}$,
\cite[Theorem 0]{McOwen1} implies that zero belongs to the continuous spectrum if $0<\delta+\frac{d}{p}-2<d-2$, it is an eigenvalue for 
$\delta+\frac{d}{p}-2<0$, and belongs to the residual spectrum for $\delta+\frac{d}{p}-2>d-2$.
\end{Rem}


The proof of Theorem \ref{th:sectorial_W} is based on Proposition \ref{prop:resovent_estimates_W} formulated and proved below.
For any given $\beta>0$ and $0<\epsilon<\pi$ consider the set of complex numbers
\[
R_{\beta,\epsilon}:=\big\{\lambda\in\C\,\big|\,|\lambda|\ge\beta,\,|\arg\lambda|\le\pi-\epsilon\big\}.
\]
For any $\lambda\in\C\setminus(-\infty,0]$ denote by $\K_\lambda$ the convolution operator
\begin{equation}\label{eq:KK_lambda}
\K_\lambda : C_c^\infty\to C^\infty,\quad\varphi\mapsto K_\lambda*\varphi,
\end{equation}
where $K_\lambda$ is the fundamental solution of the operator $(\lambda-\Delta) : \Sz'\to \Sz'$.
Denote by $d\sigma$ the volume form of the unit sphere $\s^{d-1}$ in the Euclidean space $\R^d$ and let $\Omega_{d-1}$
be the volume of $\s^{d-1}$.

\begin{Prop}\label{prop:resovent_estimates_W}
Assume that $1<p<\infty$, $\delta\in\R$ and that $m\ge 0$ is an integer. Then for any given $\beta>0$, $0<\epsilon<\pi$, and 
a multi-index $\alpha\in\Z^d_{\ge 0}$, $|\alpha|\le 2$, there exists a constant $C\equiv C_{\beta,\epsilon}>0$ (independent of $m\ge 0$) 
such that
\begin{equation}\label{eq:resolvent_W-spaces}
\big\|\partial^\alpha\big(\K_\lambda\varphi\big)\big\|_{W^{m,p}_\delta}\le
\frac{C}{|\lambda|^{1-\frac{|\alpha|}{2}}}\big\|\varphi\big\|_{W^{m,p}_\delta}
\end{equation}
for any $\varphi\in C_c^\infty$ and $\lambda\in R_{\beta,\epsilon}$.
\end{Prop}

\begin{proof}[Proof of Proposition \ref{prop:resovent_estimates_W}]
We will concentrate our attention on the case when $d\ge 3$. (The case $d=2$ is treated in the same way.)
For any $\lambda\in\C\setminus(-\infty,0]$ the fundamental solution of the operator
$(\lambda-\Delta) : \Sz'\to \Sz'$ is
\[
K_\lambda(x):=\frac{i}{4}\left(\frac{i\sqrt{\lambda}}{2\pi|x|}\right)^\mu H_\mu^{(1)}\big(i\sqrt{\lambda}|x|\big),
\quad\mu:=\frac{d}{2}-1,
\]
where $H_\mu^{(1)}(z)$ is the first Hankel function (see \cite{Taylor}).
Recall from \cite[\S 5.6]{Lebedev} that $H_\mu^{(1)}$ is holomorphic in $\C\setminus(-\infty,0]$ and
\begin{equation}\label{eq:hankel_near_zero}
H_\mu^{(1)}(z)=C_1(\mu)\frac{1}{z^\mu}\big(1+g(z^2)\big)
\end{equation}
where $C_1(\mu)$ is a constant and $g$ is a holomorphic function in an open neighborhood of $z=0$ such that $g(0)=0$. 
Moreover,  for any $0<\epsilon<\pi/2$ and for $z$ in the cone $|\arg z|\le\pi-\epsilon$ one has (\cite[\S 5.11]{Lebedev})
\begin{equation}\label{eq:hankel_at_infinity}
H_\mu^{(1)}(z)=C_2(\mu)\frac{1}{\sqrt{z}}\,e^{i z}\big(1+O(1/|z|)\big)\quad\text{as}\quad|z|\to\infty
\end{equation}
where $C_2(\mu)$ is a constant. Also, recall from \cite[\S 5.6]{Lebedev}) that
\begin{equation}\label{eq:hankel_derivative}
2\frac{d}{dz}H_\mu^{(1)}(z)=H_{\mu-1}^{(1)}(z)-H_{\mu+1}^{(1)}(z).
\end{equation}
Now, take $0<\beta<1$, $0<\epsilon<\pi/2$, and assume that $\lambda\in R_{\beta,\epsilon}$.
Then, 
\begin{equation}\label{eq:K_lambda}
K_\lambda(x)=C_\mu\lambda^\mu H\big(i\sqrt{\lambda}|x|\big)\quad\text{where}\quad
H(z):=\frac{1}{z^\mu}\,H_\mu^{(1)}(z),\quad z\in\C\setminus\R_{\le 0}.
\end{equation}
For a given cut-off function $\chi\in C_c^\infty(\R)$ such that $\chi\equiv 1$ on $[-1,1]$ and $\chi\equiv 0$ on $\R\setminus[-2,2]$
we write
\begin{equation}\label{eq:H_representation}
H(z)=C_1(\mu)\frac{1}{z^{d-2}}\chi(|z|)+H_1(z)
\end{equation}
where $H_1(z):=H(z)-C_1(\mu)\frac{1}{z^{d-2}}\chi(|z|)$ and $z\in\C\setminus\R_{\le 0}$. It follows from \eqref{eq:hankel_near_zero},
\eqref{eq:hankel_at_infinity}, and \eqref{eq:K_lambda}, that for $z\in\C\setminus\R_{\le 0}$,
\begin{equation}\label{eq:H1_near_zero}
H_1(z)=C_3(\mu)\frac{1}{z^{d-4}}\big(1+g_1(z^2)\big),
\end{equation}
where $g_1$ is a holomorphic function in an open neighborhood of $z=0$ and $g_1(0)=0$.
Moreover,
\begin{equation}\label{eq:H1_at_infinity}
H_1(z)=\frac{1}{z^\mu}H_\mu^{(1)}(z)\quad\text{for}\quad |z|\ge 2\quad\text{and}\quad z\in\C\setminus\R_{\le 0}.
\end{equation}
In view of \eqref{eq:K_lambda} and \eqref{eq:H_representation} we have
\begin{equation*}
K_\lambda(x)=K_\lambda^{(0)}(x)+K_\lambda^{(1)}(x)
\end{equation*}
where
\begin{equation}\label{eq:K(0)}
K_\lambda^{(0)}(x):=C_4(\mu)\frac{1}{|x|^{d-2}}\chi\big(\sqrt{|\lambda|}|x|\big)
\end{equation}
and
\begin{equation}\label{eq:K(1)}
K_\lambda^{(1)}(x):=C_\mu\lambda^\mu H_1\big(i\sqrt{\lambda}|x|\big)
\end{equation}
for some constants $C_4(\mu)$.
Denote by $\K_\lambda^{(0)}$ and $\K_\lambda^{(1)}$ the convolution operators corresponding to
$K_\lambda^{(0)}$ and $K_\lambda^{(1)}$. By definition,
\[
\K_\lambda=\K_\lambda^{(0)}+K_\lambda^{(1)}.
\]
First, we restrict our attention to $\K_\lambda^{(1)}$. 
Since the multiplication operator 
\begin{equation}\label{eq:multiplication_operator}
J_\delta : L^p_\delta\to L^p,\quad(J_\delta f)(x):=\x^\delta f(x),
\end{equation}
is an isometry, in order to show that $\K_\lambda^{(1)}$ is a bounded operator in $L^p_\delta$ and estimate its operator norm 
we will estimate the operator norm of $J_\delta\K_\lambda^{(1)}J_\delta^{-1}$ in $L^p$. 
To this end, we will apply the Schur test to the integral kernel of $J_\delta\K_\lambda^{(1)}J_\delta^{-1}$,
\[
K_\delta^{(1)}(x,y;\lambda):=C_\mu\lambda^\mu H_1\big(i\sqrt{\lambda}|x-y|\big)\x^\delta\y^{-\delta}.
\]
For any $x\in\R^d$ and $\lambda\in R_{\beta,\epsilon}$ we have
\begin{eqnarray}
\int_{\R^d}\big|K_\delta^{(1)}(x,y;\lambda)\big|\,dy&\le&|C_\mu||\lambda|^\mu
\int_{\R^d}\big|H_1\big(i\sqrt{\lambda}|x-y|\big)\big|\langle x-y\rangle^{|\delta|}\,dy\nonumber\\
&=&\Omega_{d-1}|C_\mu||\lambda|^\mu
\int_0^\infty\big|H_1\big(i\sqrt{\lambda}r\big)\big|(1+r^2)^{|\delta|/2}r^{d-1}\,dr\nonumber\\
&\le&\Omega_{d-1}|C_\mu|\frac{1}{|\lambda|}
\int_0^\infty\big|H_1\big(s e^{i\theta}\big)\big|\big(1+(s^2/\beta)\big)^{|\delta|/2}s^{d-1}\,ds\label{eq:schur1}
\end{eqnarray}
where we applied Lemma \ref{lem:elementary}, passed to the new variable $s=\sqrt{|\lambda|}r$ in the integral,
and then set $\theta:=\frac{\pi}{2}+\arg\sqrt{\lambda}$. Since $\lambda\in R_{\beta,\epsilon}$ we have
\begin{equation}\label{eq:theta_range}
\frac{\epsilon}{2}\le\theta\le\pi-\frac{\epsilon}{2}.
\end{equation}
It follows from \eqref{eq:hankel_at_infinity}, \eqref{eq:H1_near_zero},  and \eqref{eq:H1_at_infinity} that there exist constants 
$0<s_1<s_2<\infty$ and $C>0$ such that for any $\theta$ in \eqref{eq:theta_range},
\begin{equation}\label{eq:iniformity1}
\big|H_1\big(s e^{i\theta}\big)\big|\le
\left\{
\begin{array}{cc}
C/s^{d-4},&0\le s\le s_1,\\
C,&s_1\le s\le s_2,\\
Ce^{-s\sin\frac{\epsilon}{2}}/s^{\frac{d-1}{2}},&s_1\le s<\infty.
\end{array}
\right.
\end{equation}
This together with \eqref{eq:schur1} then implies that there exists a constant $C\equiv C_{\beta,\epsilon,d}>0$ such that
\[
\int_{\R^d}\big|K_\delta^{(1)}(x,y;\lambda)\big|\,dy\le C/|\lambda|.
\]
Similar arguments show that $\int_{\R^d}\big|K_\delta^{(1)}(x,y;\lambda)\big|\,dx$ has the same upper bound.
Hence, by the Schur test $\K_\lambda^{(1)} : L^p_\delta\to L^p_\delta$ and
\[
\big\|\K_\lambda^{(1)}\big\|_{\LL(L^p_\delta)}=\big\|J_\delta\K_\lambda^{(1)} J_\delta^{-1}\big\|_{\LL(L^p)}\le C/|\lambda|
\]
for any $\lambda\in R_{\beta,\epsilon}$. Now, take a multi-index $\alpha\in\Z^d_{\ge 0}$ such that $|\alpha|\le 2$. 
Then, $\partial^\alpha\big(\K_\lambda^{(1)}\varphi\big)=K_{\lambda,\alpha}^{(1)}*\varphi$ for any 
$\varphi\in C_c^\infty$ where
\begin{equation}\label{eq:K1_derivative}
K_{\lambda,\alpha}^{(1)}:=\partial^\alpha K_\lambda^{(1)}
\end{equation}
and $\partial^\alpha$ denotes the distributional partial derivative in $\Sz'$. One concludes from \eqref{eq:H1_near_zero}
and \eqref{eq:K(1)} that the weak partial derivative $\partial^\alpha$ in \eqref{eq:K1_derivative} coincides with 
the pointwise derivative. Then, applying the Schur test to $J_\delta K_{\lambda,\alpha}^{(1)}J_\delta^{-1}$ and arguing 
as above one concludes from \eqref{eq:hankel_near_zero}, \eqref{eq:H1_near_zero}, \eqref{eq:H1_at_infinity}, 
\eqref{eq:hankel_derivative}, and \eqref{eq:K(1)} that $\K_\lambda^{(1)} : L^p_\delta\to L^p_\delta$ and 
that there exists $C\equiv C_{\alpha,\beta,\epsilon,\delta}>0$ such that 
\begin{equation}\label{eq:K1_alpha}
\big\|\partial^\alpha\K_\lambda^{(1)}\big\|_{\LL(L^p_\delta)}\le C/|\lambda|^{1-\frac{|\alpha|}{2}}
\end{equation}
for any $\lambda\in R_{\beta,\epsilon}$. Now, consider the operator $\K_\lambda^{(0)}$. Note that by \eqref{eq:K(0)}
its kernel $K_\lambda^{(0)}$ vanishes identically at infinity and is proportional to the fundamental solution of 
the Laplacian $\Delta : \Sz'\to \Sz'$ in an open neighborhood of zero. 
For any $1\le l,j\le d$ and for any $\lambda\in R_{\beta,\epsilon}$ one easily obtains from the divergence theorem that
the following formulas for the distributional partial derivatives of $K_\lambda^{(0)}$ hold
\begin{equation}\label{eq:K0_first_derivative}
\big(\partial_l K_\lambda^{(0)}\big)(x)=(2-d)\frac{1}{|x|^{d-1}}\frac{x_l}{|x|}\chi\big(\sqrt{|\lambda|}|x|\big)+
\sqrt{|\lambda|}\frac{1}{|x|^{d-2}}\frac{x_l}{|x|}\chi'\big(\sqrt{|\lambda|}|x|\big)
\end{equation}
and 
\begin{equation}\label{eq:K0_second_derivative}
\begin{array}{ccl}
\big(\partial_j\partial_l  K_\lambda^{(0)}\big)(x)&=&
(2-d)\,\mathop{\rm P.V.}\frac{1}{|x|^d}\Big(\delta_{jl}-d\,\frac{x_j}{|x|}\frac{x_l}{|x|}\Big)\chi\big(\sqrt{|\lambda|}|x|\big)\\
&+&\sqrt{|\lambda|}\,\frac{1}{|x|^{d-1}}\Big(\delta_{jk}+(3-2d)\frac{x_j}{|x|}\frac{x_l}{|x|}\Big)\chi'\big(\sqrt{|\lambda|}|x|\big)\\
&+&|\lambda|\,\frac{1}{|x|^{d-2}}\frac{x_j}{|x|}\frac{x_l}{|x|}\chi''\big(\sqrt{|\lambda|}|x|\big)\\
&+&(2-d)\,\delta_{jl}\,\frac{\Omega_{d-1}}{d}\,\delta(x)
\end{array}
\end{equation}
where $\delta(x)$ is the Dirac delta function at zero, $\delta_{jk}$ is the Kronecker delta symbol, and
\[
\mathop{\rm P.V.}\frac{1}{|x|^d}\Big(\delta_{jl}-d\,\frac{x_j}{|x|}\frac{x_l}{|x|}\Big)\in \Sz'
\]
denotes the Cauchy principal value
\begin{equation}\label{eq:principal_value}
\lim_{\epsilon_0\to 0+}\int_{|x|\le\epsilon_0}\frac{1}{|x|^d}\Big(\delta_{jl}-d\,\frac{x_j}{|x|}\frac{x_l}{|x|}\Big)\varphi(x)\,dx
\end{equation}
for $\varphi\in \Sz$.

\begin{Rem}\label{rem:CZ-condition}
The isometric action of the orthogonal group on the unit sphere $\s^{d-1}$ easily implies that
\begin{equation}\label{eq:integral_relation}
\int_{\s^{d-1}}x_jx_l\,d\sigma=\frac{\Omega_{d-1}}{d}\,\delta_{jl}
\end{equation}
where $d\sigma$ denotes the volume form on the unit sphere. 
(In fact, for $j\ne l$ the reflection transformation $x_k\mapsto x_k$ for $k\ne j$ and $x_j\mapsto-x_j$, 
when restricted to $\s^{d-1}$, preserves the volume form on $\s^{d-1}$ and changes the sign of the integrand in \eqref{eq:integral_relation}.
This implies that \eqref{eq:integral_relation} holds for $j\ne l$.
On the other side, one has by symmetry that $\int_{\s^{d-1}}x_1^2\,d\sigma=...=\int_{\s^{d-1}}x_d^2\,d\sigma$.
Since $\int_{\s^{d-1}}\sum_{1\le k\le d}x_k^2\,d\sigma=\Omega_{d-1}$ we then obtain that \eqref{eq:integral_relation} holds for $j=l$.)
This shows that the homogeneous function of degree zero $\delta_{jl}-d\,\frac{x_j}{|x|}\frac{x_l}{|x|}$ appearing in 
\eqref{eq:principal_value} has mean value zero on $\s^{d-1}$. 
In particular, we see that the first term on the right hand side of \eqref{eq:K0_second_derivative} defines, 
by taking convolution, a SIO of Calderon-Zygmund type (see e.g. \cite{Stein}, Ch. II).
\end{Rem} 
By combining \eqref{eq:K0_first_derivative} and \eqref{eq:K0_second_derivative} with Lemma \ref{lem:schur+CZ} below 
we conclude that for any multi-index $|\alpha|\le 2$ there exists $C\equiv C_{\alpha,\beta,\epsilon,\delta}>0$ such that 
$\big\|\partial^\alpha\K_\lambda^{(0)}\big\|_{\LL(L^p_\delta)}\le C/|\lambda|^{1-\frac{|\alpha|}{2}}$
for any $\lambda\in R_{\beta,\epsilon}$.
This together with \eqref{eq:K1_alpha} then implies that for any multi-index $\alpha$ with $|\alpha|\le 2$ there exists 
a constant $C\equiv C_{\alpha,\beta,\epsilon,\delta}>0$ such that for any $\varphi\in C_c^\infty(\R)$ we have
\begin{equation}\label{eq:Klambda_alpha}
\big\|\partial^\alpha\big(\K_\lambda\varphi\big)\big\|_{L^p_\delta}\le\frac{C}{|\lambda|^{1-\frac{|\alpha|}{2}}}\|\varphi\|_{L^p_\delta}
\end{equation}
for any $\lambda\in R_{\beta,\epsilon}$.
For $\varphi\in C_c^\infty$ and any multi-index $\gamma$ with $|\gamma|\le m$ we apply \eqref{eq:Klambda_alpha} 
with $\varphi$ replaced by $\partial^\gamma\varphi$ and $\delta$ replaced by $\delta+|\gamma|$ to obtain
\[
\big\|\partial^\gamma\big[\partial^\alpha(\K_\lambda\varphi)\big]\big\|_{L^p_{\delta+|\gamma|}}\le
\frac{C}{|\lambda|^{1-\frac{|\alpha|}{2}}}\big\|\partial^\gamma\varphi\big\|_{L^p_{\delta+|\gamma|}}.
\]
By summing these inequalities over $|\gamma|\le m$ we obtain that for any $\varphi\in C_c^\infty$ we have that
\[
\big\|\partial^\alpha\big(\K_\lambda\varphi\big)\big\|_{W^{m,p}_\delta}\le
\frac{C}{|\lambda|^{1-\frac{|\alpha|}{2}}}\|\varphi\|_{W^{m,p}_\delta}
\]
for any $\lambda\in R_{\beta,\epsilon}$. This completes the proof of the Proposition.
\end{proof}

The following lemma was used in the proof of Proposition \ref{prop:resovent_estimates_W}.
For any $0<\epsilon<\pi$, $\beta>0$, $\kappa\ge 0$, and for any Lipschitz continuous function on the sphere $A : \s^{d-1}\to\R$ 
such that $\int_{\s^{d-1}}A\,d\sigma=0$ when $\kappa=0$ consider the convolution operator 
\begin{equation}\label{eq:P-convolution}
\mathcal{P}_{\lambda,\kappa} : \varphi\mapsto P_{\lambda,\kappa}*\varphi,\quad\varphi\in C_c^\infty,
\end{equation}
where
\[
P_{\lambda,\kappa}(x):=
\left\{
\begin{array}{ll}
\frac{A(x/|x|)}{|x|^{d-\kappa}}\eta\big(\sqrt{|\lambda|}|x|\big),&\kappa>0,\\
\mathop{\rm P.V.}\frac{A(x/|x|)}{|x|^d}\eta\big(\sqrt{|\lambda|}|x|\big),&\kappa=0,
\end{array}
\right.
\]
and $\eta\in C_c^\infty(\R)$ with $\eta\equiv 1$ [or $\eta\equiv 0$] in an open neighborhood of zero in $\R$.

\begin{Lem}\label{lem:schur+CZ}
Assume that $1<p<\infty$, $\delta\in\R$, and $d\ge 3$. Then, for any $0<\epsilon<\pi$, $\beta>0$, $\kappa\ge 0$, 
and for any $A : \s^{d-1}\to\R$ as described above the convolution operator 
\eqref{eq:P-convolution} extends to a bounded operator $\mathcal{P}_{\lambda,\kappa} : L^p_\delta\to L^p_\delta$ such that
$\big\|\mathcal{P}_{\lambda,\kappa}\big\|_{\LL(L^p_\delta)}\le C/|\lambda|^{\kappa/2}$
with $C>0$ independent of the choice of $\lambda\in R_{\beta,\epsilon}$.
\end{Lem}

\begin{proof}[Proof of Lemma \ref{lem:schur+CZ}]
Let us first consider the case when $\eta\equiv 1$ in an open neighborhood of zero and $\kappa=0$.
Denote $\mathcal{P}_\lambda:=\mathcal{P}_{\lambda,\kappa}|_{\kappa=0}$ and
assume that the support of $\eta$ is contained within an open ball in $\R^d$ of radius $\rho>0$.
It then follows from the second statement of Lemma \ref{lem:elementary} that
\[
\x^\delta/\y^\delta=1+O(|x-y|)
\]
with constant independent of $\lambda\in R_{\beta,\epsilon}$ and $x,y\in\R^d$ with $|x-y|<\rho/\beta $.
By combining this with the fact that $\eta\big(\sqrt{|\lambda|}|x-y|\big)=0$ for 
$|x-y|\ge\rho/\sqrt{|\lambda|}$ we conclude that for any $\varphi\in C_c^\infty$,
\begin{eqnarray}
\big(J_\delta\mathcal{P}_\lambda J_\delta^{-1}\varphi\big)(x)&=&\lim_{\epsilon_0\to 0+}
\int_{|y|\ge\epsilon_0}\frac{A(x-y/|x-y|)}{|x-y|^d}\eta\big(\sqrt{|\lambda|}|x-y|\big)\x^\delta\y^{-\delta}\varphi(y)\,dy\nonumber\\
&=&\lim_{\epsilon_0\to 0+}\int_{|y|\ge\epsilon_0}\frac{A(x-y/|x-y|)}{|x-y|^d}\eta\big(\sqrt{|\lambda|}|x-y|\big)
\big(1+O(|x-y|)\big)\varphi(y)\,dy\nonumber\\
&=&\lim_{\epsilon_0\to 0+}\int_{|y|\ge\epsilon_0}\frac{A(x-y/|x-y|)}{|x-y|^d}\eta\big(\sqrt{|\lambda|}|x-y|\big)\varphi(y)\,dy+
\big(\mathcal{Q}_{\lambda,\delta}\varphi\big)(x)\label{eq:CZ-part}
\end{eqnarray}
where $\mathcal{Q}_{\lambda,\delta}$ has an integral kernel $Q_{\lambda,\delta}(x-y)$ such that
\[
\big|Q_{\lambda,\delta}(x-y)\big|\le M\,\frac{|A(x-y/|x-y|)|}{|x-y|^{d-1}}\eta\big(\sqrt{|\lambda|}|x-y|\big) 
\]
with a constant $M>0$ independent on $x,y\in\R^d$ and $\lambda\in R_{\beta,\epsilon}$.
Note that $Q_{\lambda,\delta}\in L^1$ and
\begin{align}
\|Q_{\lambda,\delta}\|_{L^1}\le M\int_{\R^d}\frac{|A(x/|x|)|}{|x|^{d-1}}\eta\big(\sqrt{|\lambda|}|x|\big)\,dx=
\frac{M}{|\lambda|^{1/2}}\int_{\R^d}\frac{|A(y/|y|)|}{|y|^{d-1}}\eta(|y|)\,dy
\end{align}
where we pass to the variable $y=\sqrt{|\lambda|} x$ in the integral. Then, by Young's inequality we obtain that
there exists a constant $C>0$ such that 
\begin{equation}\label{eq:Q_lambda}
\|\mathcal{Q}_{\lambda,\delta}\|_{\LL(L^p)}\le C/|\lambda|^{1/2}
\end{equation}
for any $\lambda\in R_{\beta,\epsilon}$. 
Let us now consider the first expression appearing on the right side of \eqref{eq:CZ-part},
\begin{equation}\label{eq:T_lambda}
(\mathcal{T}_\lambda\varphi)(x):=
\lim_{\epsilon_0\to 0+}\int_{|y|\ge\epsilon_0}\frac{A(x-y/|x-y|)}{|x-y|^d}\eta\big(\sqrt{|\lambda|}|x-y|\big)\varphi(y)\,dy.
\end{equation}
By changing the variables $y':=\sqrt{|\lambda|}y$ and $x':=\sqrt{|\lambda|}x$ we see that
\begin{equation}\label{eq:T-conjugation}
\mathcal{T}_{\lambda}\varphi=\big(\mathcal{H}_\lambda\mathcal{T}\mathcal{H}_\lambda^{-1}\big)\varphi
\end{equation}
where
\begin{equation}\label{eq:T1}
\big(\mathcal{T}\varphi\big)(x'):=\int_{\R^d}\frac{A(x'-y'/|x'-y'|)}{|x'-y'|^d}\eta\big(|x'-y'|\big)\varphi(y')\,dy'
\end{equation}
and $\big(\mathcal{H}_\lambda\varphi\big)(x):=\varphi\big(\sqrt{|\lambda|}x\big)$.
It follows from Lemma \ref{lem:claderon-zygmund} that the transformation \eqref{eq:T1} extends 
to a bounded operator $\mathcal{T} : L^p\to L^p$ with norm independent of $\lambda\in R_{\beta,\epsilon}$.
Since $\|\mathcal{H}_\lambda\|_{\LL(L^p)}= 1/|\lambda|^{d/2p}$ and since $\mathcal{H}_\lambda^{-1}=\mathcal{H}_{\lambda^{-1}}$
we obtain from \eqref{eq:T-conjugation} that
\[
\big\|\mathcal{T}_{\lambda}\big\|_{\LL(L^p)}\le
\big\|\mathcal{H}_\lambda\big\|_{\LL(L^p)}\big\|\mathcal{T}\big\|_{\LL(L^p)}\big\|\mathcal{H}_\lambda^{-1}\big\|_{\LL(L^p)}
=\big\|\mathcal{T}\big\|_{\LL(L^p)}
\]
which proves that the expression in \eqref{eq:T_lambda} defines a bounded operator in 
$L^p$ with norm independent of $\lambda\in R_{\beta,\epsilon}$. 
By combining this with \eqref{eq:Q_lambda} and the fact that \eqref{eq:multiplication_operator} is an isometry we conclude 
that $\mathcal{P}_\lambda : L^p_\delta\to L^p_\delta$ is a bounded operator such that 
$\|\mathcal{P}_\lambda\|_{\LL(L^p_\delta)}\le C$ with a constant $C>0$ independent of $\lambda\in R_{\beta,\epsilon}$.
This completes the proof of the lemma in the case when $\kappa=0$.
The case when $\eta\equiv 1$ in a neighborhood of zero and $\kappa>0$ as well as the case when 
$\eta\equiv 0$ in an open neighborhood of zero follow easily form Young's inequality and the arguments used 
to prove \eqref{eq:Q_lambda}.
\end{proof}

Now we are ready to prove Theorem \ref{th:sectorial_W}.

\begin{proof}[Proof of Theorem \ref{th:sectorial_W}]
Take $\omega>0$ and $0<\epsilon<\pi$. By choosing $\beta>0$ sufficiently small we see that
$\Sigma_{\omega,\epsilon}\subseteq R_{\beta,\epsilon}$. In particular, with an appropriate choice of the constant $C>0$, 
the estimate \eqref{eq:resolvent_W-spaces} in Proposition \ref{prop:resovent_estimates_W} holds uniformly in 
$\lambda\in\Sigma_{\omega,\epsilon}$ and $|\alpha|\le 2$.
In particular, or any $\lambda\in\Sigma_{\omega,\epsilon}$ the convolution operator \eqref{eq:KK_lambda}
extends to a bounded linear map 
\begin{equation}\label{eq:KK_lambda'}
\K_\lambda : W^{m,p}_\delta\to{\widetilde W}^{m+2,p}_\delta.
\end{equation}
Since $\K_\lambda$ is the fundamental solution of $\lambda-\Delta : \Sz'\to\Sz'$ for $\lambda\in\C\setminus(-\infty,0])$ we have that
$(\lambda-\Delta)\big(\K_\lambda\varphi\big)=\varphi$ for any test function $\varphi\in C_c^\infty$.
By the continuity of \eqref{eq:KK_lambda'} we then conclude that the bounded linear map
$(\lambda-\Delta) : {\widetilde W}^{m+2,p}_\delta\to W^{m,p}_\delta$ is onto. 
By the open mapping theorem and the fact that $\lambda-\Delta : \Sz'\to\Sz'$ for $\lambda\in\C\setminus(-\infty,0])$ is injective
we then see that $(\lambda-\Delta) : {\widetilde W}^{m+2,p}_\delta\to W^{m,p}_\delta$ is an isomorphism of Banach spaces 
for any $\lambda\in\Sigma_{\omega,\epsilon}$. In particular, this implies that the spectrum of $\Delta|_{W^{m,p}_\delta}$ is
contained in $(-\infty,0]$ and that the map \eqref{eq:KK_lambda'} is the resolvent $(\lambda-\Delta)^{-1}$ of $\Delta|_{W^{m,p}_\delta}$.
This proves (i). By Proposition \ref{prop:resovent_estimates_W},
\[
\big\|\partial^\alpha(\lambda-\Delta)^{-1}\big\|_{\LL(W^{m,p}_{\delta})}\le C/|\lambda|^{1-\frac{|\alpha|}{2}},\quad|\alpha|\le 2\,.
\]
By combining this with the fact that there exists $C_1\equiv C_1(\epsilon)>0$ such that $|\lambda-\omega|\le C_1|\lambda|$ 
for any $\lambda\in\Sigma_{\omega,\epsilon}$, we conclude the proof of (ii).
\end{proof}

Recall from \cite{Pazy,RenardyRogers} that by definition, a family $\{T(t)\}_{t\ge 0}$ of bounded linear maps of a Banach space $X$ is called
an {\em analytic semigroup on $X$ with angle $\vartheta\in (0,\pi/2]$} if the map $[0,\infty)\to\LL(X)$, $t\mapsto T(t)$, can be extended to 
the complex sector (cf. \eqref{eq:conic_sets})
\begin{equation}\label{eq:S-sector}
{\bbS}_\vartheta\equiv\big\{z\in\C\,\big|\, z\ne 0, |\arg z|<\vartheta\big\}
\end{equation}
such that: i) the map $z\mapsto T(z)$ is analytic from ${\bbS}_\vartheta$ to the space $\LL(X)$ of bounded linear operators on $X$; 
ii) $T(z)$ is a semigroup, i.e. $T(z_1+z_2)=T(z_1) T(z_2)$ for any $z_1,z_2\in {\bbS}_\vartheta$; iii) $S(z)$ converges strongly to 
the identity operator on $X$ as $z\to 0$ in ${\bbS}_\vartheta\cup\{0\}$. 
We have the following

\begin{Th}\label{th:SonW}
Assume that $1<p<\infty$, $m\ge 0$, and $\delta\in \R$. Then, $\{S(t)\}_{t\geq 0}$ is an analytic semigroup on $W_\delta^{m,p}$ 
with angle $\vartheta=\pi/2$ and generator $\Delta|_{W^{m,p}_\delta}$ with domain ${\widetilde W}^{m+2,p}_\delta$. 
The subspace $W^{m+2,p}_\delta\subseteq{\widetilde W}^{m+2,p}_\delta$ is invariant with respect to 
$\{S(t)\}_{t\geq 0}$ and for any $u_0\in W^{m+2,p}_\delta$ the heat equation \eqref{eq:heat} (with $\nu=1$) has a unique solution
$u\in C\big([0,\infty),W^{m+2,p}_\delta\big)\cap C^1\big([0,\infty),W^{m,p}_\delta\big)$ and $u(t)=S(t)u_0$ for $t\in[0,\infty)$.
\end{Th}

\begin{Rem}\label{rem:SonW}
It follows from Theorem \ref{th:SonW} that for any $u_0\in W^{m+2,p}_\delta$ and $\vartheta\in(0,\pi/2)$
the holomorphic map $u : \bbS_{\vartheta}\to W^{m+2,p}_\delta$ when extended by continuity to $\bbS_{\vartheta}\cup\{0\}$ 
satisfies the (complex) heat equation $u_z=\Delta u$, $u|_{t=0}=u_0$.
\end{Rem}

\begin{proof}[Proof of Theorem \ref{th:SonW}]
Assume that $1<p<\infty$, $m\ge 0$, and $\delta\in \R$. 
The theorem follows directly from the resolvent estimates in Theorem \ref{th:sectorial_W}.
Indeed, take an arbitrary $\vartheta\in(0,\pi/2)$ and set $\alpha\equiv 0$ and $\epsilon\equiv\pi/2-\vartheta$.
Then, by Theorem \ref{th:sectorial_W} (ii), the operator $\Delta|_{W^{m,p}_\delta}$ with domain 
${\widetilde W}^{m+2,p}_\delta$ is sectorial on $\Sigma_{\omega,\vartheta+\pi/2}$ and hence, 
by \cite[Theorem 5.2, \S2.5]{Pazy}, generates an analytic semigroup $\{T(t)\}_{t\ge 0}$ on $W^{m,p}_\delta$ 
with angle $\vartheta\in(0,\pi/2)$. In particular, the space $W^{m,p}_\delta$ is invariant with respect to the semigroup and
for any initial data $u_0\in{\widetilde W}^{m+2,p}_\delta$ the heat equation \eqref{eq:heat} (with $\nu=1$) has a unique solution
$u(t):=T(t)u_0$, $t\ge 0$, such that
\begin{equation}\label{eq:heat_solution_W}
u\in C\big([0,\infty),{\widetilde W}^{m+2,p}_\delta\big)\cap C^1\big([0,\infty),W^{m,p}_\delta\big).
\end{equation}
Then, by Corollary \ref{coro:exp=S}, we obtain that $S(t)\equiv T(t)$ for any $t\ge 0$. Since the angle $\vartheta\in(0,\pi/2)$
can be chosen arbitrarily close to $\pi/2$, we conclude the proof of the first statement of the corollary.
Since the space $W^{m+2,p}_\delta$ is invariant with respect to the semigroup and by \eqref{eq:inclusions_W} 
it is contained in ${\widetilde W}^{m+2,p}_\delta$ we obtain from \eqref{eq:heat_solution_W} that
for any $u_0\in W^{m+2,p}_\delta$ the heat equation \eqref{eq:heat} (with $\nu=1$) has a unique solution
\[
u\in C\big([0,\infty),W^{m+2,p}_\delta\big)\cap C^1\big([0,\infty),W^{m,p}_\delta\big).
\]
This completes the proof of Theorem \ref{th:SonW}.
\end{proof}

Theorem \ref{th:sectorial_W}, Theorem \ref{th:SonW}, and the arguments used in the proof of \cite[Theorem 1.1.3 (iii)]{LLMP})
imply the following corollary.

\begin{Coro}\label{coro:S(t)-estimatesW}
Assume that $1<p<\infty$, $m\ge 0$, and $\delta\in \R$.
Then, for any $\omega>0$ and for any $\vartheta\in(0,\pi/2)$ there exists $C\equiv C_{\omega,\vartheta}>0$ such that 
for any multi-index $\alpha\in\Z^d_{\ge 0}$, $|\alpha|\le 2$, 
\begin{equation}\label{eq:S(t)-W-estimates}
\big\|\partial^\alpha S(z)\big\|_{\LL(W^{m,p}_\delta)}\le
C e^{\omega|z|}/|z|^{\frac{|\alpha|}{2}}\,.
\end{equation}
for any $z\in\bbS_\vartheta$. Moreover, for any multi-index $\alpha$ with $|\alpha|\le 2$
the map $z\mapsto\partial^\alpha S(z)$, $\bbS_{\pi/2}\to\LL(W^{m,p}_\delta)$, is holomorphic.
\end{Coro}

\begin{Rem}
In fact, one can prove that the exponential factor $e^{\omega|z|}$ in the estimate \eqref{eq:S(t)-W-estimates} can be replaced
by a polynomial factor of the form $(1+|z|)^\mu$ for some $\mu\in\R$ and a constant $C>0$ independent of the choice of $\omega>0$. 
Since we will not need this type of estimate we will not prove it here.
\end{Rem}

\begin{proof}[Proof of Corollary \ref{coro:S(t)-estimatesW}]
Assume that $1<p<\infty$, $m\ge 0$, and $\delta\in \R$.
Take $\omega>0$, $\vartheta\in(0,\pi/2)$, and a multi-index $\alpha$ such that $|\alpha|\le 2$.
We will first show that the map 
\begin{equation}\label{eq:p-resolventW}
\lambda\mapsto\partial^\alpha(\lambda-\Delta)^{-1},\quad\Sigma_{\omega,\pi/2-\vartheta}\to\LL(W^{m,p}_\delta),
\end{equation}
is holomorphic. In order to see this take $\varphi\in C^\infty_c$ and note that
$\partial^\alpha(\lambda-\Delta)^{-1}\varphi=(\lambda-\Delta)^{-1}\partial^\alpha\varphi$.
This implies that the map $\lambda\mapsto\partial^\alpha(\lambda-\Delta)^{-1}\varphi$, 
$\Sigma_{\omega,\pi/2-\vartheta}\to W^{m,p}_\delta$, is holomorphic by Theorem \ref{th:sectorial_W} (i) 
and the properties of the resolvent. Hence, for any functional $f\in\big(W^{m,p}_\delta\big)'$ the map 
\[
\lambda\mapsto\big\langle f,\partial^\alpha(\lambda-\Delta)^{-1}\varphi\big\rangle,\quad
\Sigma_{\omega,\pi/2-\vartheta}\to\C,
\]
is holomorphic. Since by Theorem  \ref{th:sectorial_W} (ii) the map \eqref{eq:p-resolventW} is locally bounded and since
$C^\infty_b$ is dense in $W^{m,p}_\delta$, we then conclude from \cite[Theorem 3.12, Ch. III]{Kato} that 
\eqref{eq:p-resolventW} is holomorphic.

Now, assume that $t>0$ and set 
\begin{equation}\label{eq:eta}
\eta:=\pi-\vartheta'\quad\text{\rm  for some}\quad\vartheta'\in(\vartheta,\pi/2). 
\end{equation}
For a given $\rho>0$ consider the contour 
\[
\Gamma_{\rho,\eta}:=\big\{r e^{-i\eta}\,\big|-\infty\le r\le \rho\big\}\sqcup
\big\{\rho e^{i\alpha}\,\big|\,-\eta\le\alpha\le\eta\big\}\sqcup
\big\{r e^{i\eta}\,\big|\rho\le r\le\infty\big\}
\]
and the integral
\begin{equation}\label{eq:riesz-exponent1}
\frac{1}{2\pi i}\int_{\omega+\Gamma_{\rho,\eta}}e^{t\lambda}\,\big[\partial^\alpha(\lambda-\Delta)^{-1}\big]\,d\lambda=
\frac{e^{\omega t}}{2\pi i}\int_{\Gamma_{\rho,\eta}}e^{t\lambda}\,\big[\partial^\alpha\big((\omega+\lambda)-\Delta\big)^{-1}\big]\,d\lambda\,.
\end{equation}
It follows easily from Theorem \ref{th:sectorial_W} (ii) and the fact that $\eta\in(\pi/2,\pi)$ that the integral on the right side of 
\eqref{eq:riesz-exponent1} converges as an improper Riemann integral in $\LL(W^{m,p}_\delta)$. In particular, we see that
for $t>0$,
\begin{equation}\label{eq:pS-integral}
\partial^\alpha S(t)=
\frac{e^{\omega t}}{2\pi i}\int_{\Gamma_{\rho,\eta}}e^{t\lambda}\,\big[\partial^\alpha\big((\omega+\lambda)-\Delta\big)^{-1}\big]\,d\lambda
\end{equation}
where 
\[
S(t)\equiv T(t):=\frac{1}{2\pi i}\int_{\omega+\Gamma_{\rho,\eta}}e^{t\lambda}\,(\lambda-\Delta)^{-1}\,d\lambda
\]
is the analytic semigroup generated by the sectorial operator $\Delta|_{W^{m,p}_\delta}$ (cf. the proof of  Corollary \ref{coro:S(t)-estimatesW}).
Passing to the variable $\zeta:=t\lambda$ in the integral on the right side of \eqref{eq:riesz-exponent1} we obtain
\begin{equation}\label{eq:riesz-exponent2}
\frac{e^{\omega t}}{2\pi i}\int_{\Gamma_{t\rho,\eta}}e^\zeta\,\big[\partial^\alpha\big((\omega+(\zeta/t))-\Delta\big)^{-1}\big]\,\frac{d\zeta}{t}
=\frac{e^{\omega t}}{2\pi i}\int_{\Gamma_{\rho,\eta}}e^\zeta\,\big[\partial^\alpha\big((\omega+(\zeta/t))-\Delta\big)^{-1}\big]\,\frac{d\zeta}{t}
\end{equation}
where we used the analyticity of the integrand to replace the contour $\Gamma_{t\rho,\eta}$ with $\Gamma_{\rho,\eta}$.
One then obtains from Theorem \ref{th:sectorial_W} (ii) that
\begin{align}
\Big\|\int_{\Gamma_{\rho,\eta}}\!\!\!\!\!e^\zeta\,\big[\partial^\alpha\big((\omega+(\zeta/t))-\Delta\big)^{-1}\big]\,d\zeta\Big\|_{\LL(W^{m,p}_\delta)}\!\!\!\!
&\le\int_{\Gamma_{\rho,\eta}}|e^\zeta|\Big\|\partial^\alpha\big((\omega+(\zeta/t))-\Delta\big)^{-1}\Big\|_{\LL(W^{m,p}_\delta)}\,d\zeta\nonumber\\
&\le C t^{|\alpha|/2}\Big(2\int_\rho^\infty e^{r\cos\eta}\,\frac{dr}{r}+\frac{1}{\rho}\int_{-\eta}^\eta e^{\rho\cos\alpha}\,d\alpha\Big)\label{eq:riesz-exponent3}
\end{align}
where $C\equiv C_{\omega,\vartheta}>0$ is coming from the estimate in Theorem \ref{th:sectorial_W} (ii)
and the integrals on the right side converge. It then follows from \eqref{eq:riesz-exponent1}, \eqref{eq:pS-integral}, \eqref{eq:riesz-exponent2},
and \eqref{eq:riesz-exponent3}, that there exists $C\equiv C_{\omega,\theta}$ such that
\[
\big\|\partial^\alpha S(t)\big\|_{\LL(W^{m,p}_\delta)}\le
C e^{\omega t}/t^{\frac{|\alpha|}{2}}
\]
for any $t>0$.
Now, we replace the variable $t$ on the right side of \eqref{eq:riesz-exponent2} by $z$ and allow $z\in\bbS_\vartheta$. 
It follows from \eqref{eq:eta} that $\omega+\zeta/z\in\Sigma_{\omega,\vartheta'-\vartheta}$ for $\zeta\in\Gamma_{\rho,\eta}$.
This allows us to apply Theorem \ref{th:sectorial_W} (ii) and conclude as in \eqref{eq:riesz-exponent3} that the estimate
\eqref{eq:S(t)-W-estimates} holds for any $z\in\bbS_\vartheta$. The argument also imply that one can differentiate with respect
to the complex parameter $z\in\bbS_\vartheta$ under the sign of the integral to conclude that
\[
z\mapsto\partial^\alpha S(z),\quad\bbS_\vartheta\to\LL(W^{m,p}_\delta),
\]
is holomorphic. Since $\vartheta\in(0,\pi/2)$ can be chosen arbitrarily close to $\pi/2$ we conclude the proof of the corollary.
\end{proof}

By combining Corollary \ref{coro:S(t)-estimatesW} with Lemma \ref{lem:equivalent_W-norms} we obtain

\begin{Coro}\label{coro:S(t)-estimatesW*}
Assume that $1<p<\infty$, $m\ge 0$, and $\delta\in\R$.
Then, for any $\omega>0$ and for any $\vartheta\in(0,\pi/2)$ there exists $C\equiv C_{\omega,\vartheta}>0$ such that  
for any $\tau\in\{0,1,2\}$,
\begin{equation}\label{eq:S(t)-W-estimates*}
\big\|S(z)\big\|_{\LL(W^{m,p}_{\delta+\tau},W^{m+\tau,p}_\delta)}\le
C\,e^{\omega|z|}\max\big(1,1/|z|^{\frac{\tau}{2}}\big)
\end{equation}
for any $z\in\bbS_\vartheta$. Moreover, the map $z\mapsto S(z)$, 
$\bbS_{\pi/2}\to\LL(W^{m,p}_{\delta+\tau},W^{m+\tau,p}_\delta)$, is holomorphic.
\end{Coro}

\section{The heat semigroup on asymptotic spaces}\label{sec:A-spaces}
Assume that $1<p<\infty$ and let $m\ge 0$ and $n,N\ge 0$ with $0\le n\le N$ be integers.
In this section we show that $\{S(t)\}_{t\ge 0}$ extends to an analytic semigroup with angle $\vartheta=\pi/2$
on $\A^{m,p}_{n,N;0}$. For simplicity of notation we will denote in this section the real spaces and their complexifications
by the same symbols.

Recall that by definition $v\in\A^{m,p}_{n,N;0}$ is represented in a unique way as $v=a+f$ where $a$ is its asymptotic part
and $f$ is the reminder (see \eqref{AWlog-expansion1}-\eqref{AW-expansion3}). Consider the Banach space
\begin{eqnarray}\label{def:D(Lap|A)}
{\widetilde\A}^{m+2,p}_{n,N;0}:=
\big\{v\in\Sz'\,\big|\,\partial^\alpha v\in\A^{m,p}_{n,N;0}, |\alpha|\le 2\big\}
\end{eqnarray}
equipped with the norm $\|v\|_{{\widetilde\A}^{m+2,p}_{n,N;0}}:=\sum_{|\alpha|\le 2}\|\partial^\alpha v\|_{\A^{m,p}_{n,N;0}}$.
Note that $\A^{m+2,p}_{n,N;0}$ is a dense subspace in ${\widetilde\A}^{m,p}_{n,N;0}$ and by \eqref{def:D(Lap|A)} we have
\[
\A^{m+2,p}_{n,N;0}\subseteq{\widetilde\A}^{m+2,p}_{n,N;0}\subseteq \A^{m,p}_{n,N;0}
\] 
with bounded inclusion maps. We will first prove the following theorem.

\begin{Th}\label{th:sectorial_A}
Assume that $1<p<\infty$ and let $m\ge 0$ and $0\le n\le N$ be integers. Then, we have
\begin{itemize}
\item[(i)] For any $\lambda\in\C\setminus(-\infty,0]$ the map 
\begin{equation*}
(\lambda-\Delta) : {\widetilde\A}^{m+2,p}_{n,N;0}\to\A^{m,p}_{n,N;0}
\end{equation*}
is an isomorphism of Banach spaces.
In particular, the (distributional) Laplacian $\Delta$, when considered as an unbounded operator $\Delta|_{\A^{m,p}_{n,N;0}}$ in 
$\A^{m,p}_{n,N;0}$ with domain $D(\Delta|_{\A^{m,p}_{n,N;0}})={\widetilde\A}^{m+2,p}_{n,N;0}$, is closed and
its spectrum is contained in $(-\infty,0]$. 
\item[(ii)] For any $\omega>0$ and for any $0<\epsilon<\pi$ there exists a 
positive constant $C\equiv C_{\omega,\epsilon,m}>0$ such that
\[
\big\|\partial^\alpha(\lambda-\Delta)^{-1}\big\|_{\LL(\A^{m,p}_{n,N;0})}\le C/|\lambda-\omega|^{1-\frac{|\alpha|}{2}}
\]
for any $\lambda\in\Sigma_{\omega,\epsilon}$ and any multi-index $\alpha\in\Z_{\ge 0}^d$ with $|\alpha|\le 2$.
\end{itemize}
\end{Th}

Let $\chi\in C_c^\infty$, $\chi\equiv 1$ on $[-1,1]$, $\chi\equiv 1$ on $\R\setminus(-2,2)$, $\chi\ge 0$, be
the cut-off function appearing in the definition of the asymptotic space $\A^{m,p}_{n,N;0}$.
It follows from the definition of $\A^{m,p}_{n,N;0}$ that the asymptotic term
\begin{equation}\label{eq:asymptotic_term}
\chi(r) a(\theta)(\log r)^\ell/r^k
\end{equation}
with $n\le k\le N$ and $0\le\ell\le k$ belongs to $\A^{m,p}_{n,N;0}$ if and only if $a\in H^{m+1+(N-k),p}(\s^{d-1})$.
Assume that $\lambda\in\C\setminus(-\infty,0]$ and recall that $\lambda-\Delta : \Sz'\to\Sz'$ is an isomorphism.
We will first study the action of $(\lambda-\Delta)^{-1} : \Sz'\to\Sz'$ on the asymptotic term \eqref{eq:asymptotic_term}.
We start with a preparation.
For any $a\in C^\infty(\s^{d-1})$ denote $\hat{a}(x):=a(x/|x|)$, $\hat{a}\in C^\infty\big(\R^d\setminus\{0\}\big)$, and for any given $1\le j\le d$ 
consider the following operator on the unit sphere $\s^{d-1}\equiv\big\{x\in\R^d\,\big|\,|x|=1\big\}$ in $\R^d$,
\begin{equation}\label{eq:d-hat}
C^\infty(\s^{d-1})\to C^\infty(\s^{d-1}),\quad{\hat\partial}_j : a\mapsto\big(\partial_j\hat{a}\big)\big|_{x=\theta\in \s^{d-1}}.
\end{equation}
One easily sees (e.g. in polar coordinates) that ${\hat\partial}_j$ is a differential operator on $\s^{d-1}$ of order one with $C^\infty$-smooth coefficients.
In geometrical terms, for any given $\theta\in \s^{d-1}$ the operator ${\hat\partial}_j$ at $\theta$ corresponds to the orthogonal projection of 
the vector field $\partial_j\in T_\theta\R^d$ onto the the tangent space $T_\theta \s^{d-1}$ to the sphere. One easily checks that
\begin{equation}\label{eq:d-hat_relation}
\partial_j\big[a\big(x/|x|\big)\big]=\frac{1}{r}\big({\hat\partial}_ja\big)\big(x/|x|\big),\quad \forall x\in\R^d\setminus\{0\}.
\end{equation}
We have the following lemma.

\begin{Lem}\label{lem:resolvent_on_asymptotics}
Take $0\le n\le k\le N$, $0\le\ell\le k$ and assume that $a\in H^{m+1+(N-k),p}(\s^{d-1})$ where $m>1+\frac{d}{p}$. 
Then for any $\lambda\in\C\setminus(-\infty,0]$ and for any given integer $j_N\ge 0$ such that $k+2j_N\le N$ one has
\begin{eqnarray}\label{eq:j_N-general}
\lambda\,\big(\lambda-\Delta\big)^{-1}\Big(\chi(r)\,\frac{a(\theta)(\log r)^\ell}{r^k}\Big)&=&
\sum_{0\le j\le j_N}\frac{1}{\lambda^j}\chi(r)\,\frac{\sum_{0\le l\le\ell}(\log r)^l\big(\mathcal{P}_{2j}^l a\big)(\theta)}{r^{k+2j}}\nonumber\\
&+&\frac{1}{\lambda^{j_N}}\!\big(\lambda-\Delta\big)^{-1}\!\!
\left(\chi(r)\,\frac{\sum_{0\le l\le\ell}(\log r)^l\big(\mathcal{P}_{2j_N+2}^l a\big)(\theta)}{r^{k+2j_N+2}}\right)\\
&+&\big(\lambda-\Delta\big)^{-1}\big(\mathcal{R}^\ell_{2j_N+1}a\big)\nonumber
\end{eqnarray}
where $\mathcal{P}_{2j}^l$ is a differential operator of order $\,\le 2j$ on the sphere $\s^{d-1}$ with $C^\infty$-smooth coefficients 
independent of $\lambda$ and
\begin{equation}\label{eq:R-operator}
\big(\mathcal{R}^\ell_{2j_N+1}a\big)(r,\theta):=\sum_{0\le j\le j_N}\frac{1}{\lambda^j}A^\ell_{2j_N+1,j}(r)\big(B^\ell_{2j_N+1,j}a\big)(\theta)
\end{equation} 
where $A^\ell_{2j_N+1,j}\in C_c^\infty(\R)$ has support in the interval $[1,2]$ and $B^\ell_{2j_N+1,j}$ is a differential operator of order 
$\le 2j_N+1$ on $\s^{d-1}$ with $C^\infty$-smooth coefficients independent of $\lambda$. 
The differential operators $\mathcal{P}^l_{2j_N+2}$ and $\mathcal{R}^\ell_{2j_N+1}$ appearing above depend on the choice of $n\le k\le N$ 
and $0\le\ell\le k$.
\end{Lem}

\begin{Rem}\label{rem:smoothness}
By Lemma  \ref{lem:properties_W} in Appendix \ref{sec:appendix_properties} for $j_N=\big[(N-k)/2\big]$ where $[\cdot]$ denotes 
the integer part of a real number,
\[
\mathcal{P}^l_{2j_N+2}a\in 
\left\{
\begin{array}{l}
H^{m-1,p}(\s^{d-1}),\quad(N-k)\,\text{-- even},\\ 
H^{m,p}(\s^{d-1}),\quad(N-k)\,\,\text{-- odd},
\end{array}
\right. 
\]
and
\[
\mathcal{R}^\ell_{2j_N+1}a\in W^{m,p}_\delta
\]
for any given weight $\delta\in\R$. In particular, we see that $\mathcal{R}^\ell_{2j_N+1}a$ belongs to the remainder space
$W^{m,p}_{\gamma_N}$ (cf. \eqref{def:gamma_N}).
\end{Rem}

\begin{proof}[Proof of Lemma \ref{lem:resolvent_on_asymptotics}]
Recall that in polar coordinates $\Delta=\partial_r^2+\frac{d-1}{r}\partial_r+\frac{1}{r^2}\Delta_S$ where 
$r\ne 0$ and $\Delta_S$ denotes the Laplace-Beltrami operator on the unit sphere $\s^{d-1}$. 
Under the assumption $r\ne 0$, this easily implies (cf. \cite[\S 3]{McOwenTopalov3}) that for $\ell\ge 2$,
\begin{equation}\label{eq:Delta(l>= 2)}
\begin{array}{l}
\Delta\Big(\frac{a(\theta) (\log r)^\ell}{r^k}\Big)
=\frac{\ell(\ell-1) a(\theta) (\log r)^{\ell-2}-\ell(2k+2-d)a(\theta) (\log r)^{\ell-1}+
\big((\Delta_S a)(\theta)+k(k+2-d) a(\theta)\big)(\log r)^\ell}
{r^{k+2}}\,,
\end{array}
\end{equation}
while for $\ell=1$ and $\ell=0$ one has respectively,
\begin{equation}\label{eq:Delta(l=1)}
\Delta\Big(\frac{a(\theta)\log r}{r^k}\Big)=
\frac{-(2k+2-d)a(\theta)+\big((\Delta_S \, a)(\theta)+k(k+2-d) a(\theta)\big)\log r}{r^{k+2}}
\end{equation}
and
\begin{equation}\label{eq:Delta(l=0)}
\Delta\Big(\frac{a(\theta)}{r^k}\Big)=\frac{(\Delta_S \,a)(\theta)+k(k+2-d) a(\theta)}{r^{k+2}}.
\end{equation}
It follows from  \eqref{eq:Delta(l>= 2)},  \eqref{eq:Delta(l=1)}, and  \eqref{eq:Delta(l=0)} that for $\ell\ge 0$,
\[
\big(\lambda-\Delta\big)\Big(\chi\,\frac{a(\theta) (\log r)^\ell}{r^k}\Big)=
\lambda\cdot\chi\,\frac{a(\theta) (\log r)^\ell}{r^k}-\sum_{0\le l\le\ell}\chi\,\frac{(\log r)^l\big(\mathcal{P}_2^l a\big)(\theta)}{r^{k+2}}
-\mathcal{R}^\ell_1a
\]
where\footnote{Note that for $k\ge d-2$, $\lambda_k'$ is the $k'\equiv k-2+d$ eigenvalue of $(-\Delta_S)$ on $\s^{d-1}$}
$\mathcal{P}^\ell_2 a=(\Delta_S+\lambda_k')a$, $\lambda_k':=k(k+2-d)$, and
\begin{equation}\label{eq:R_1}
\big(\mathcal{R}^\ell_1a\big)(r,\theta):=
2\nabla\chi\cdot\nabla\Big(\frac{a(\theta)(\log r)^\ell}{r^k}\Big)+\Delta\chi\cdot\frac{a(\theta)(\log r)^\ell}{r^k}.
\end{equation}
This implies that for $\lambda\in\C\setminus(-\infty,0]$,
\begin{eqnarray}\label{eq:j_N=0}
\lambda\,\big(\lambda-\Delta\big)^{-1}\Big(\chi\,\frac{a(\theta)(\log r)^\ell}{r^k}\Big)&=&\chi\,\frac{a(\theta) (\log r)^\ell}{r^k}+
\big(\lambda-\Delta\big)^{-1}\left(\chi\,\frac{\sum_{0\le l\le\ell}(\log r)^l\big(\mathcal{P}_2^l a\big)(\theta)}{r^{k+2}}\right)\nonumber\\
&+&\big(\lambda-\Delta\big)^{-1}\big(\mathcal{R}^\ell_1a\big)\,.
\end{eqnarray}
It follows from \eqref{eq:R_1} and \eqref{eq:d-hat_relation} that $\mathcal{R}^\ell_1a$ is of the form \eqref{eq:R-operator}. 
This proves \eqref{eq:j_N-general} in the case when $j_N=0$.
The general case of \eqref{eq:j_N-general} then follows by applying inductively \eqref{eq:j_N=0} to the second term
on the right hand side of \eqref{eq:j_N-general}.
\end{proof}

As a consequence of the results above we obtain the following important

\begin{Prop}\label{prop:resovent_estimates_A}
Assume that $1<p<\infty$ and consider the asymptotic space $\A^{m,p}_{n,N;0}$
for some $m\ge 0$ and $0\le n\le N$ integer. Then, for any given $\beta>0$, $0<\epsilon<\pi$, 
there exists a constant $C\equiv C_{\beta,\epsilon,m}>0$ such for any multi-index $\alpha\in\Z^d_{\ge 0}$, 
$|\alpha|\le 2$, 
\begin{equation}\label{eq:resolvent_A-spaces}
\big\|\partial^\alpha(\lambda-\Delta)^{-1} u\big\|_{\A^{m,p}_{n,N;0}}\le
\frac{C}{|\lambda|^{1-\frac{|\alpha|}{2}}}\big\|u\big\|_{\A^{m,p}_{n,N;0}}
\end{equation}
for any $u\in\A^{m,p}_{n,N;0}$ and $\lambda\in R_{\beta,\epsilon}$.
\end{Prop}

\begin{proof}[Proof of Proposition \ref{prop:resovent_estimates_A}]
The proposition follows from Theorem \ref{th:sectorial_W}, Lemma \ref{lem:resolvent_on_asymptotics}, and
Lemma \ref{lem:properties_W}. More specifically, let us first consider the case when $\alpha=0$. Take
an arbitrary $u\in\A^{m,p}_{n,N;0}$ and let
\[
\chi(r) a(\theta)(\log r)^\ell/r^k,\quad a\in H^{m+1+(N-k),p}(\s^{d-1}),
\]
with $n\le k\le N$ and $0\le\ell\le k$ be an asymptotic term in $u$.
Then by Lemma \ref{lem:resolvent_on_asymptotics} the expansion \eqref{eq:j_N-general} holds with $j_N=\big[(N-k)/2\big]$.
Note that by Remark \ref{rem:smoothness} and Lemma \ref{lem:properties_W} $(i)$,
\begin{equation}\label{eq:P_(2j_N+2)}
\chi(r)\,\frac{\sum_{0\le l\le\ell}(\log r)^l\big(\mathcal{P}_{2j_N+2}^l a\big)(\theta)}{r^{k+2j_N+2}}\in
\left\{
\begin{array}{l}
W^{m-1,p}_{\gamma_{N}+1},\quad(N-k)\,\text{-- even},\\ 
W^{m,p}_{\gamma_N},\quad(N-k)\,\,\text{-- odd},
\end{array}
\right. 
\end{equation}
and 
\begin{equation}\label{eq:R_(2j_N+1)}
\mathcal{R}^\ell_{2j_N+1}a\in W^{m,p}_{\gamma_N}.
\end{equation}
This together with Theorem \ref{th:sectorial_W} then implies that the last two terms on the right hand
side of \eqref{eq:j_N-general} belong to $W^{m,p}_{\gamma_N}$ when $(N-k)$ is odd, or
to ${\widetilde W}^{m+1,p}_{\gamma_N+1}$ when $(N-k)$ is even.
Since by \eqref{eq:inclusions_W},
\[
{\widetilde W}^{m+1,p}_{\gamma_N+1}\subseteq W^{m,p}_{\gamma_N},
\]
we conclude that the last two terms on the right hand side of \eqref{eq:j_N-general} belong to 
$W^{m,p}_{\gamma_N}$. Hence, by \eqref{eq:j_N-general},
\[
\lambda\big(\lambda-\Delta\big)^{-1}\Big(\chi(r)\,\frac{a(\theta)(\log r)^\ell}{r^k}\Big)\in 
\A^{m,p}_{n,N;0}
\]
and, by Theorem \ref{th:sectorial_W}, Lemma \ref{lem:properties_W}, 
and the definition of the norm in $\A^{m,p}_{n,N;0}$,
there exists $C\equiv C_{\beta,\epsilon,m}>0$ such that
\[ 
\Big\|\lambda\big(\lambda-\Delta\big)^{-1}\Big(\chi(r)\,\frac{a(\theta)(\log r)^\ell}{r^k}\Big)\Big\|_{\A^{m,p}_{n,N;0}}\le
C\,\|a\|_{H^{m+1+(N-k)}}
\]
for any $a\in H^{m+1+(N-k)}$ and for any $\lambda\in R_{\beta,\epsilon}$.
By combining this with Theorem \ref{th:sectorial_W}, applied to the remainder of $u$, we conclude that
\eqref{eq:resolvent_A-spaces} holds in the case when $\alpha=0$.
Let us now assume that $\alpha\ne 0$.
First, note that by \eqref{eq:d-hat_relation}, for any $b\in H^{m+1+(N-s),p}(\s^{d-1})$, $0\le s\le N$, $l\ge 0$,
and $1\le j\le d$,
\[
\partial_j\Big(\frac{b(\theta)(\log r)^l}{r^s}\Big)=
\frac{\big(({\hat\partial}_jb)(\theta)-s\,\theta_jb(\theta)\big)(\log r)^l+l\,\theta_j b(\theta)(\log r)^{l-1}}{r^{s+1}}
\]
where $\theta_j\equiv x_j/|x|$. This together with the definition \eqref{def:AW;-norm} of the norm in $\A^{m,p}_{n,N;0}$ implies that
\begin{equation}\label{eq:diff_assympt_term}
\left\|\partial_j\Big(\frac{b(\theta)(\log r)^l}{r^s}\Big)\right\|_{\A^{m,p}_{n,N;0}}\le C\|b\|_{H^{m+1+(N-s)}}.
\end{equation}
Let us first assume that $(N-k)$ is odd. Then, the terms in \eqref{eq:P_(2j_N+2)} and \eqref{eq:R_(2j_N+1)}
belong to $W^{m,p}_{\gamma_N}$. This and Theorem \ref{th:sectorial_W} then imply that
the last two terms on the right hand side of \eqref{eq:j_N-general} belong to $W^{m,p}_{\gamma_N}$
together with their weak partial derivatives up to order two. By combining this with \eqref{eq:diff_assympt_term} 
we obtain from \eqref{eq:j_N-general} that for any $n\le k\le N$, $0\le\ell\le k$, and $|\alpha|\le 2$,
\[
\lambda\,\partial^\alpha\big(\lambda-\Delta\big)^{-1}\Big(\chi(r)\,\frac{a(\theta)(\log r)^\ell}{r^k}\Big)\in
\A^{m,p}_{n,N;0}
\] 
and there exists $C\equiv C_{\alpha,\beta,\epsilon,m}>0$ such that
\begin{equation}\label{eq:resolvent_asymptotic_terms}
\Big\|\lambda\,\partial^\alpha\big(\lambda-\Delta\big)^{-1}\Big(\chi(r)\,\frac{a(\theta)(\log r)^\ell}{r^k}\Big)\Big\|_{\A^{m,p}_{n,N;0}}\le
C\,\|a\|_{H^{m+1+(N-k)}}
\end{equation}
for any $a\in H^{m+1+(N-k)}$ and for any $\lambda\in R_{\beta,\epsilon}$. We then take $u\in\A^{m,p}_{n,N;0}$ and apply 
\eqref{eq:resolvent_asymptotic_terms} to the asymptotic terms of $u$ and Theorem \ref{th:sectorial_W} to the remainder of $u$ 
to conclude that \eqref{eq:resolvent_A-spaces} holds for any $|\alpha|\le 2$ and $\lambda\in R_{\beta,\epsilon}$, in the case when $(N-k)$ is odd.
If $(N-k)$ is even we consider \eqref{eq:j_N-general} with $j_N=\big[(N-k)/2\big]-1$. Note that in this case
$\mathcal{R}^\ell_{2j_N+1}a\in W^{m,p}_{\gamma_N}$ and
\[
\chi(r)\,\frac{\sum_{0\le l\le\ell}(\log r)^l\big(\mathcal{P}_{2j_N+2}^l a\big)(\theta)}{r^{k+2j_N+2}}\in W^{m+1,p}_{\gamma_N-1}.
\]
Hence, 
\begin{equation*}
\mathcal{R}^\ell_{2j_N+1}a\in W^{m,p}_{\gamma_N}\quad\text{\rm and}\quad
\partial^\alpha\left(\chi(r)\,\frac{\sum_{0\le l\le\ell}(\log r)^l\big(\mathcal{P}_{2j_N+2}^l a\big)(\theta)}{r^{k+2j_N+2}}\right)\in
\left\{
\begin{array}{cc}
W^{m-1,p}_{\gamma_{N}+1},&|\alpha|=2,\\ 
W^{m,p}_{\gamma_N},&|\alpha|=1.
\end{array}
\right. 
\end{equation*}
We then apply $\partial^\alpha$ to \eqref{eq:j_N-general} with $j_N=\big[(N-k)/2\big]-1$, 
and argue as in the case $\alpha=0$ (cf. \eqref{eq:P_(2j_N+2)} and \eqref{eq:R_(2j_N+1)}) to conclude that 
the estimate \eqref{eq:resolvent_A-spaces} holds for any $0<|\alpha|\le 2$, $\lambda\in R_{\beta,\epsilon}$, and $(N-k)$ even. 
This completes the proof of Proposition \ref{prop:resovent_estimates_A}.
\end{proof}

\medskip

With this preparation let us now return to Theorem \ref{th:sectorial_A}.

\begin{proof}[Proof of Theorem \ref{th:sectorial_A}]
The theorem follows from Proposition \ref{prop:resovent_estimates_A} and the arguments used to prove Theorem \ref{th:SonW}.
\end{proof}

As a consequence we obtain

\begin{Th}\label{th:SonA}
Assume that $1<p<\infty$ and let $m\ge 0$ and $0\le n\le N$ be integers. Then, 
$\{S(t)\}_{t\ge 0}$ is an analytic semigroup of bounded operators on $\A_{n,N;0}^{m,p}$
with angle $\vartheta=\pi/2$ and generator $\Delta|_{\A_{n,N;0}^{m,p}}$ with domain ${\widetilde\A}^{m+2,p}_{n,N;0}$. 
The subspace $\A^{m+2,p}_{n,N;0}\subseteq{\widetilde\A}^{m+2,p}_{n,N;0}$ is invariant with respect to $\{S(t)\}_{t\ge 0}$
and for any $u_0\in\A^{m+2,p}_{n,N;0}$ the heat equation \eqref{eq:heat} (with $\nu=1$) has a unique solution
$u\in C\big([0,\infty),\A^{m+2,p}_{n,N;0}\big)\cap C^1\big([0,\infty),\A^{m,p}_{n,N;0}\big)$ and $u(t)=S(t)u_0$ for $t\in[0,\infty)$. 
\end{Th}

\begin{Rem}\label{rem:SonA}
It follows from Theorem \ref{th:SonA} that for any $u_0\in\A^{m+2,p}_{n,N;0}$ and $\vartheta\in(0,\pi/2)$ the holomorphic map
$u : \bbS_{\vartheta}\to\A^{m+2,p}_{n,N;0}$ when extended by continuity to $\bbS_{\vartheta}\cup\{0\}$
satisfies the (complex) heat equation $u_z=\Delta u$, $u|_{t=0}=u_0$.
\end{Rem}


\begin{proof}[Proof of Theorem \ref{th:SonA}]
The theorem follows from the resolvent estimates in Theorem \ref{th:sectorial_A} and the argument used in the proof of
Theorem \ref{th:SonW}.
\end{proof}

Theorem \ref{th:sectorial_A}, Theorem \ref{th:SonA}, and the arguments used in the proof of
Corollary \ref{coro:S(t)-estimatesW}.

\begin{Coro}\label{coro:S(t)-estimatesA}
Assume that $1<p<\infty$ and let $m\ge 0$ and $0\le n\le N$ be integers.
Then, for any $\omega>0$ and for any $\vartheta\in(0,\pi/2)$ there exists $C\equiv C_{\omega,\vartheta,m}>0$ such that 
for any multi-index $\alpha\in\Z^d_{\ge 0}$, $|\alpha|\le 2$, 
\begin{equation}\label{eq:S(t)-A-estimates}
\big\|\partial^\alpha S(z)\big\|_{\LL(\A^{m,p}_{n,N;0})}\le
C\,e^{\omega|z|}/|z|^{\frac{|\alpha|}{2}}
\end{equation}
for any $z\in\bbS_\vartheta$. Moreover, for any multi-index $\alpha$ with $|\alpha|\le 2$
the map $z\mapsto\partial^\alpha S(z)$, $\bbS_{\pi/2}\to\LL(\A^{m,p}_{n,N;0})$, is holomorphic.
\end{Coro}

\begin{proof}[Proof of Corollary \ref{coro:S(t)-estimatesA}]
We will follow the proof of Corollary \ref{coro:S(t)-estimatesW}.
Assume that $1<p<\infty$ and let $m\ge 0$ and $0\le n\le N$ be integers.
Take $\omega>0$, $\vartheta\in(0,\pi/2)$, and a multi-index $\alpha$ such that $|\alpha|\le 2$.
We will first show that the map 
\begin{equation}\label{eq:p-resolventA}
\lambda\mapsto\partial^\alpha(\lambda-\Delta)^{-1},\quad\Sigma_{\omega,\pi/2-\vartheta}\to\LL(\A^{m,p}_{n,N;0}),
\end{equation}
is holomorphic. In order to see this take $\varphi\in\A^{m+|\alpha|,p}_{n,N;0}$ and note that by Theorem \ref{th:sectorial_A} (i) and
Proposition \ref{pr:A-properties},
\[
\partial^\alpha(\lambda-\Delta)^{-1}\varphi=(\lambda-\Delta)^{-1}\partial^\alpha\varphi\in\A^{m,p}_{n+|\alpha|,N+|\alpha|;-|\alpha|}
\subseteq\A^{m,p}_{n,N;0}.
\]
This implies that the map $\lambda\mapsto\partial^\alpha(\lambda-\Delta)^{-1}\varphi$, 
$\Sigma_{\omega,\pi/2-\vartheta}\to\A^{m,p}_{n,N;0}$, is holomorphic by Theorem \ref{th:sectorial_A} (i) 
and the properties of the resolvent. Hence, for any functional $f\in\big(\A^{m,p}_{n,N;0}\big)'$ the map 
\[
\lambda\mapsto\big\langle f,\partial^\alpha(\lambda-\Delta)^{-1}\varphi\big\rangle,\quad
\Sigma_{\omega,\pi/2-\vartheta}\to\C,
\]
is holomorphic. Since by Theorem  \ref{th:sectorial_A} (ii) the map \eqref{eq:p-resolventW} is locally bounded and since
$\A^{m+|\alpha|,p}_{n,N;0}$ is dense in $\A^{m,p}_{n,N;0}$, we then conclude from \cite[Theorem 3.12, Ch. III]{Kato} that 
\eqref{eq:p-resolventA} is analytic. 
Now, one can argue in precisely the same way as in the proof of Corollary \ref{coro:S(t)-estimatesA} to conclude 
the proof of the corollary.
\end{proof}

By combining Corollary \ref{coro:S(t)-estimatesA} with Lemma \ref{lem:equivalent_A-norms} we obtain

\begin{Coro}\label{coro:S(t)-estimatesA*}
Assume that $1<p<\infty$ and let $m\ge 0$ and $0\le n\le N$ be integers.
Then, for any $\omega>0$ and for any $\vartheta\in(0,\pi/2)$ there exists $C\equiv C_{\omega,\vartheta,m}>0$ such that  
for any $\tau\in\{0,1,2\}$, 
\begin{equation}\label{eq:S(t)-A-estimates*}
\big\|S(z)\big\|_{\LL(\A^{m,p}_{n+\tau,N+\tau;0},\A^{m+\tau,p}_{n,N;0})}\le
C\,e^{\omega|z|}\max\big(1,1/|z|^{\frac{\tau}{2}}\big)
\end{equation}
for any $z\in\bbS_\vartheta$. Moreover, the map $z\mapsto S(z)$, 
$\bbS_{\pi/2}\to\LL(\A^{m,p}_{n+\tau,N+\tau;0},\A^{m+\tau,p}_{n,N;0})$, is holomorphic.
\end{Coro}

\section{The Navier-Stokes flow on asymptotic spaces}\label{sec:navier-stokes}
Assume that $m>2+\frac{d}{p}$, $N\in\Z_{\ge 0}$, and $1<p<\infty$. In this section we prove the local well-posedness of the 
Navier-Stokes equation  \eqref{eq:NS} in the asymptotic space with log terms $\A_{N;0}^{m,p}\equiv\A_{0,N;0}^{m,p}$.
For simplicity of notation, in what follows we write $\A^{m,p}_{n,N;\ell}$ regardless of whether we mean the function space of
scalar functions or tensor fields on $\R^d$ with components in $\A^{m,p}_{n,N;\ell}$. 
Recall that $\accentset{\,\,\,\circ}\A^{m,p}_{n,N;\ell}$ denotes the closed space of divergence free vector fields in
$\A^{m,p}_{n,N;\ell}$,
\[
\accentset{\,\,\,\circ}\A^{m,p}_{n,N;\ell}\equiv\big\{u\in\A^{m,p}_{n,N;\ell}\,\big|\,\Div u=0\big\}.
\]

\begin{Rem}\label{rem:definition_solution}
Let $I$ be one of the intervals $[0,T]$ or $[0,T)$ for some $T>0$.
We call $u\in C\big(I,\accentset{\,\,\,\circ}\A^{m,p}_{N;0}\big)\cap C^1\big(I,\accentset{\,\,\,\circ}\A^{m-2,p}_{N;0}\big)$ a {\em solution} of
the Navier-Stokes equation if $u$ satisfies \eqref{eq:NS} for some $\p\in C\big(I,\Sz'\big)$ such that for any $t\in I$ 
the gradient $\nabla\p(t)\in L^\infty_\delta$ for a given $\delta>0$ (independent of $t$) where $L^\infty_\delta$ denotes the space of 
measurable vector fields $v$ on $\R^d$ such that $\x^\delta|v(x)|$ is (essentially) bounded on $\R^d$.
We will see below that for any $t\in I$ the gradient of the pressure $\nabla\p(t)$ (and the pressure $\p(t)$ itself) has an asymptotic expansion 
with log terms in the spatial direction (see \eqref{eq:trick6}, \eqref{eq:p-asymptotics}). 
The decay condition on the gradient of $\p(t)$ is necessary since otherwise the uniqueness statement in 
Theorem \ref{th:NS} does {\em not} hold.
\end{Rem}

Now, take $u_0\in\accentset{\,\,\,\circ}\A^{m,p}_{N;0}$. 
Our first step is to re-write the Navier-Stokes equation in a convenient form (see e.g. \cite{McOwenTopalov3}): 
Let $u\in C\big([0,T],\accentset{\,\,\,\circ}\A^{m,p}_{N;0}\big)\cap C^1\big([0,T],\accentset{\,\,\,\circ}\A^{m-2,p}_{N;0}\big)$ be a solution of \eqref{eq:NS}. 
Then
\begin{equation}\label{eq:NS1}
u_t+u\cdot\nabla u=\nu \Delta u-\nabla\p,
\end{equation}
where for any given $t\in[0,T]$ the gradient $\nabla\p$ satisfies the decay condition in Remark \ref{rem:definition_solution}. 
By applying the divergence operator $\Div : \Sz'\to \Sz'$ to both sides of \eqref{eq:NS1} and then using that 
$\Div u=0$ we conclude that pointwise in $t\in[0,T]$,
\begin{equation}\label{eq:trick1}
\Div\big(u\cdot\nabla u\big)=-\Delta\p.
\end{equation}
In view of the Sobolev embedding $\accentset{\,\,\,\circ}\A^{m,p}_{N;0}\subseteq  C^2$ one obtains by a direct computation that
\begin{eqnarray}
\Div\big(u\cdot\nabla u\big)&=&\tr\big([du]^2\big)+u\cdot\nabla(\Div u)\label{eq:trick2}\\
&=&\tr\big([du]^2\big)\label{eq:trick3},
\end{eqnarray}
where $[du]^2$ is the square of the Jacobi matrix $[du]$ of $u(t,\cdot) : \R^d\to\R^d$ and $\tr$ is the trace of a matrix.
Denote,
\[
Q(u):=\tr\big([du]^2\big).
\]
By \eqref{eq:trick1} and \eqref{eq:trick3} we conclude that pointwise in $t\in[0,T]$,
\begin{equation}\label{eq:trick4}
-\Delta\p=Q(u)\quad\text{and}\quad-\Delta\circ\nabla\p=\nabla\circ Q(u).
\end{equation}
It follows from Proposition \ref{pr:A-properties} (ii) and Proposition \ref{pr:A-products} (ii) that 
$[du]\in\A^{m-1}_{1,N+1;-1}$ and $\tr[du]^2\in\A^{m-1,p}_{2,N+2;-2}$. Hence,
\begin{equation}\label{eq:trick5}
Q(u)\equiv\tr[du]^2\in\A^{m-1,p}_{2,N+2;-2}\quad\text{and}\quad\nabla\circ Q(u)\in\A^{m-2,p}_{3,N+3;-3}.
\end{equation}
Now, assume that $0<\gamma_0+(d/p)<1$.
Recall from \cite[Proposition 3.1]{McOwenTopalov3} that there is a closed subspace ${\widehat\A}^{m,p}_{1,N+1;-1}$ 
in $\A^{m,p}_{1,N+1;0}$ such that
\begin{equation}\label{eq:Delta_isomorphism}
\Delta : {\widehat\A}^{m,p}_{1,N+1;-1}\to\A^{m-2,p}_{3,N+3;-3}
\end{equation}
is an isomorphism.
(We refer to \cite[Section 3]{McOwenTopalov3} for the definition of the space ${\widehat\A}^{m,p}_{1,N+1;-1}$ and 
the proof that \eqref{eq:Delta_isomorphism} is an isomorphism.) This together with \eqref{eq:trick4}, \eqref{eq:trick5}, 
and the decay condition on $\nabla\p$ implies that for any $t\in[0,T]$,
\begin{equation}\label{eq:trick6}
\nabla\p(t)=\Delta^{-1}\circ\nabla\circ Q(u(t))\in{\widehat\A}^{m,p}_{1,N+1;-1}\subseteq \A^{m,p}_{1,N+1;0}\subseteq \A^{m,p}_{N;0},
\end{equation}
where $\Delta^{-1}$ is the inverse of \eqref{eq:Delta_isomorphism}.
The same holds also in the case when $\gamma_0+(d/p)=0$ -- see the proof of \cite[Proposition C.1]{McOwenTopalov4}.
Here we also used the standard fact that the kernel of the Laplacian $\Delta : \Sz'\to \Sz'$ consists of 
harmonic polynomials in $\R^d$. In this way we see that $u$ satisfies in $\A^{m-2,p}_{N;0}$ the equation
\begin{equation}\label{eq:NS-modified}
u_t+u\cdot\nabla u=\nu \Delta u+\Delta^{-1}\circ\nabla\circ Q(u).
\end{equation}
In fact, we have

\begin{Prop}\label{prop:NS-modified}
Assume that $m>2+\frac{d}{p}$, $N\in\Z_{\ge 0}$, and $1<p<\infty$.
A curve $u\in C\big([0,T],\A^{m,p}_{N;0}\big)\cap C^1\big([0,T],\A^{m-2,p}_{N;0}\big)$ is a solution of
\eqref{eq:NS-modified} with initial data $u_0\in\accentset{\,\,\,\circ}\A^{m,p}_{N;0}$ if and only if $u$ lies in 
$C\big([0,T],\accentset{\,\,\,\circ}\A^{m,p}_{N;0}\big)\cap C^1\big([0,T],\accentset{\,\,\,\circ}\A^{m-2,p}_{N;0}\big)$ and is a solution of
the Navier-Stokes equation \eqref{eq:NS}.
\end{Prop}

\begin{Rem}\label{rem:NS-modified}
In what follows we prove that equation \eqref{eq:NS-modified} is locally well-posed in $\A^{m,p}_{N;0}$, i.e.
for any $u_0\in\A^{m,p}_{N;0}$ there exist $T>0$ and a unique solution 
$u\in C\big([0,T],\A^{m,p}_{N;0}\big)\cap C^1\big([0,T],\A^{m-2,p}_{N;0}\big)$ that depends continuously on the initial data 
$u_0\in\A^{m,p}_{N;0}$. By Proposition \ref{prop:NS-modified} we will then conclude that for a divergence free $u_0$
the solution $u$ stays divergent free for all $t\in[0,T]$, 
$u\in C\big([0,T],\accentset{\,\,\,\circ}\A^{m,p}_{N;0}\big)\cap C^1\big([0,T],\accentset{\,\,\,\circ}\A^{m-2,p}_{N;0}\big)$, and solves 
the Navier-Stokes equation.
\end{Rem}

\begin{proof}[Proof of Proposition \ref{prop:NS-modified}]
It remains to prove only the direct implication. Without loss of generality we will assume in \eqref{AW-expansion3} that 
$0<\gamma_0+(d/p)<1$; the case when $\gamma_0+(d/p)=0$ is treated in the same way as in \cite[Proposition C.1]{McOwenTopalov4}.
Let
\begin{equation}\label{eq:u-regularity}
u\in C\big([0,T],\A^{m,p}_{N;0}\big)\cap C^1\big([0,T],\A^{m-2,p}_{N;0}\big)
\end{equation}
be a solution of \eqref{eq:NS-modified} with initial data $u_0\in\accentset{\,\,\,\circ}\A^{m,p}_{N;0}$.
By applying the divergence operator $\Div$ to both sides of \eqref{eq:NS-modified} and then using
\eqref{eq:trick2} and the fact that $\Delta^{-1}$ commutes with the partial derivative $\partial/\partial x_j$ for
$1\le j\le d$ we conclude that
\begin{equation}\label{eq:div-equation}
(\Div u)_t+u\cdot\nabla (\Div u)=0.
\end{equation}
In view of \eqref{eq:u-regularity} and the Sobolev embedding $\A^{m,p}_{0,N;0}\subseteq  C^2$ we conclude that 
$u\in C^1\big([0,T]\times\R^d,\R^d\big)$. Moreover, since $[du]\in C\big([0,T],\A^{m-1,p}_{1,N+1;-1}\big)$ one 
sees from the definition of the asymptotic space $\A^{m-1,p}_{1,N+1;-1}$ that the (uniform) matrix norm of 
the Jacobi matrix $[du](t,x)$ is bounded uniformly in $(t,x)\in[0,T]\times\R^d$, i.e. $\sup_{[0,T]\times\R^d}\big|[du]\big|\le L$ 
for some positive constant $L>0$. 
This implies that the vector field $u$ is globally Lipschitz on $\R^d$ uniformly in $t\in[0,T]$.
Hence, for any $t\in[0,T]$ there is a uniquely defined diffeomorphism $\varphi(t,\cdot) : \R^d\to\R^d$ so that
$\varphi\in C^1\big([0,T]\times\R^d,\R^d\big)$,
\[
\varphi_t=u\big(t,\varphi(t,x)\big),
\]
and $\varphi(0,x)=x$ for any $x\in\R^d$.
This together with \eqref{eq:div-equation} then implies that for any $t\in[0,T]$ and for any $x\in\R^d$,
\[
(\Div u)\big(t,\varphi(t,x)\big)=(\Div u_0)(x)=0.
\]
Since for any $t\in[0,T]$ the map $\varphi(t,\cdot) : \R^d\to\R^d$ is a diffeomorphism, we conclude from the equality above that $\Div u(t)=0$ 
for any $t\in[0,T]$. Hence, for any $t\in[0,T]$ the vector field $u(t)$ is divergence free and 
$u\in C\big([0,T],\accentset{\,\,\,\circ}\A^{m,p}_{N;0}\big)\cap C^1\big([0,T],\accentset{\,\,\,\circ}\A^{m-1,p}_{N;0}\big)$. 
As above (cf. \eqref{eq:trick5} and \eqref{eq:trick6}) we obtain from Proposition \ref{pr:A-properties} and Proposition \ref{pr:A-products} that
pointwise in $t\in[0,T]$,
\[
Q(u)\in\A^{m-1,p}_{2,N+2;-2}\subseteq\A^{m-2,p}_{2,N+2;-2}\quad\text{\rm and}\quad
\Delta^{-1}\circ\nabla\circ Q(u)\in\A^{m,p}_{1,N+1;-1}\,.
\]
By setting 
\begin{equation}\label{eq:p-asymptotics}
\p(t):=\Delta^{-1}\circ Q(u(t))\in{\widehat\A}^{m,p}_{0,N;0}
\end{equation} 
where $\Delta^{-1}$ is the inverse of the map $\Delta : {\widehat\A}^{m,p}_{0,N;0}\to\A^{m-2,p}_{2,N+2;-2}$ considered in 
\cite[Remark 3.3]{McOwenTopalov3} we conclude that $u$ is a solution of the Navier-Stokes equation.
\end{proof}

Now, we re-write \eqref{eq:NS-modified} in the form
\begin{equation}\label{eq:NS-modified'}
u_t=\nu\Delta u+F(u)
\end{equation}
where
\begin{equation}\label{eq:NS-modified'_nonlinearity}
F(u):=\Delta^{-1}\circ\nabla\circ Q(u)-u\cdot\nabla u.
\end{equation}
As discussed above, Proposition \ref{pr:A-properties} and Proposition \ref{pr:A-products} imply that
$u\cdot\nabla u\in\A^{m-1,p}_{1,N+1;-1}$ and 
$\Delta^{-1}\circ\nabla\circ Q(u)\in{\widehat\A}^{m,p}_{1,N+1;-1}\subseteq \A^{m,p}_{1,N+1;0}$
for $u\in\A^{m,p}_{N;0}$. Moreover, it follows from \eqref{eq:trick5}, Proposition \ref{pr:A-products}, and
the fact that \eqref{eq:Delta_isomorphism} is an isomorphism, that the map
\begin{equation}\label{eq:F}
F : \A^{m,p}_{N;0}\to\A^{m-1,p}_{1,N+1;0}
\end{equation}
is analytic. In particular, we see that \eqref{eq:F} is locally Lipschitz. 
We will first prove the following stronger statement. 

\begin{Lem}\label{lem:lipschitz}
Assume that $m>2+\frac{d}{p}$, $N\in\Z_{\ge 0}$, and $1<p<\infty$.
There exists a positive constant $\kappa>0$ such that for any $\rho>0$ and for any $u,v\in B_{\A^{m,p}_{N;0}}(\rho)$,
\[
\|F(u)-F(v)\|_{\A^{m-1,p}_{1,N+1;0}}\le \kappa\rho\|u-v\|_{\A^{m,p}_{N;0}}.
\]
\end{Lem}

\begin{Rem}\label{rem:lipschitz}
The map \eqref{eq:F} extends by analyticity to a polynomial map
\begin{equation}\label{eq:F-complexified}
F : \A^{m,p}_{N;0,\C}\to\A^{m-1,p}_{1,N+1;0,\C}
\end{equation}
where $X_\C$ denotes the compexification of a given real Banach space $X$. 
One easily sees from the proof below that Lemma \ref{lem:lipschitz} will also hold
for \eqref{eq:F-complexified} in the corresponding complexified spaces.
\end{Rem}

\begin{proof}[Proof of Lemma \ref{lem:lipschitz}]
We write $F(u)=F_1(u)+F_2(u)$ where $F_1(u):=-u\cdot\nabla u$ and $F_2(u):=\Delta^{-1}\circ\nabla\circ Q(u)$.
By Proposition \ref{pr:A-properties} and Proposition \ref{pr:A-products}, the map
\[
F_1 : \A^{m,p}_{N;0}\to\A^{m-1,p}_{1,N+1;-1}\subseteq \A^{m-1,p}_{1,N+1;0}
\]
is well-defined and analytic. The differential of this map at a given $u\in\A^{m,p}_{N;0}$ is
\[
d_uF_1 : \A^{m,p}_{N;0}\to\A^{m-1,p}_{1,N+1;0},\quad \delta u\mapsto u\cdot\nabla(\delta u)+(\delta u)\cdot\nabla u.
\]
Using again Proposition \ref{pr:A-properties} and Proposition \ref{pr:A-products} (ii) we estimate
\begin{eqnarray}
\|(d_uF_1)(\delta u)\|_{\A^{m-1,p}_{1,N+1;0}}&\le&\kappa_1\|u\|_{\A^{m,p}_{N;0}}\|\nabla(\delta u)\|_{\A^{m-1,p}_{1,N+1;-1}}+
\kappa_2\|\delta u\|_{\A^{m,p}_{N;0}}\|\nabla u\|_{\A^{m-1,p}_{1,N+1;-1}}\nonumber\\
&\le& \kappa'\|u\|_{\A^{m,p}_{N;0}}\|\delta u\|_{\A^{m,p}_{N;0}}\label{eq:lipschitz1}
\end{eqnarray}
where $\kappa_1$, $\kappa_2$, and $\kappa'$ are positive constants independent of the choice of $u,\delta u\in\A^{m,p}_{N;0}$.
By arguing in a similar way and by using the continuity of the inverse of \eqref{eq:Delta_isomorphism}, one proves that
\begin{equation}\label{eq:lipschitz2}
\|(d_uF_2)(\delta u)\|_{\A^{m-1,p}_{1,N+1;0}}\le\kappa''\|u\|_{\A^{m,p}_{N;0}}\|\delta u\|_{\A^{m,p}_{N;0}},
\end{equation}
where $\kappa''>0$ is independent of the choice of $u,\delta u\in\A^{m,p}_{N;0}$. By combining \eqref{eq:lipschitz1} and \eqref{eq:lipschitz2} we
obtain that there exists a positive constant $\kappa>0$ so that for any $u\in\A^{m,p}_{N;0}$,
\begin{equation}\label{eq:lipschitz3}
\|d_u F\|_{\mathcal{L}(\A^{m,p}_{N;0},\A^{m-1,p}_{1,N+1;0})}\le \kappa\|u\|_{\A^{m,p}_{N;0}}
\end{equation}
where $\|\cdot\|_{\mathcal{L}(X,Y)}$ denotes the uniform norm in the space of bounded linear maps ${\mathcal{L}(X,Y)}$ between
two Banach spaces $X$ and $Y$.
Finally, the statement of the lemma follows from \eqref{eq:lipschitz3} by standard arguments involving the mean-value theorem.
\end{proof}

We are now ready to prove Theorem \ref{th:NS}. In fact, we will first prove a stronger result.
Recall from Remark \ref{rem:lipschitz} that $X_\C$ denotes the complexification of a given real Banach space $X$.
For a given $T>0$ consider the conic set $\bbS_{\vartheta,T}$ in $\C$ as defined in \eqref{eq:conic_sets}. 
We will say that $u\in C\big(\bbS_{\vartheta,T},X_\C\big)$ is {\em continuous at zero} if it extends to
a continuous map $\bbS_{\vartheta,T}\cup\{0\}\to X_\C$. For the simplicity of notation we will denote the extended map 
by the same letter $u\in C\big(\bbS_{\vartheta,T}\cup\{0\},X_\C\big)$.

\begin{Def}\label{def:NS-holomorphic_solution}
A holomorphic map $u\in C\big(\bbS_{\vartheta,T},\accentset{\,\,\,\circ}{\A}^{m,p}_{N;0,\C}\big)$
that is continuous at zero and $u(0)=u_0$ is called a {\em holomorphic solution of the Navier-Stokes equation on $\bbS_{\vartheta,T}$} if
\[
u|_{[0,T)}\in C\big([0,T),\accentset{\,\,\,\circ}{\A}^{m,p}_{N;0}\big)\cap C^1\big([0,T),\accentset{\,\,\,\circ}{\A}^{m-2,p}_{N;0}\big)
\]
is a solution of the Navier-Stokes equation \eqref{eq:NS} (see Remark \ref{rem:definition_solution}).
\end{Def}

\noindent Recall that $C_b(D,Y)$, where $D$ is a metric space and $Y$ is a Banach space, denotes the Banach space of 
bounded continuous maps $D\to Y$ equipped with the $\sup$-norm. We have the following theorem.

\begin{Th}\label{th:NS-complex}
Assume that $m>2+\frac{d}{p}$, $N\in\Z_{\ge 0}$, and $1<p<\infty$. Then, for any $\rho>0$ and 
for any choice of the angle $\vartheta\in(0,\pi/2)$ there exists $T\equiv T(\rho,\vartheta)>0$ so that 
for any divergence free vector field $u_0\in B^{m,p}_{N;0}(\rho)$ there exists a holomorphic solution of 
the Navier-Stokes equation on $\bbS_{\vartheta,T}$, $u\in C\big(\bbS_{\vartheta,T},\accentset{\,\,\,\circ}{\A}^{m,p}_{N;0,\C}\big)$,
that is bounded and depends Lipschitz continuously on the initial data in the sense that the data-to-solution map, $u_0\mapsto u$, 
$B^{m,p}_{N;0}(\rho)\cap\accentset{\,\,\,\circ}{\A}^{m,p}_{N;0}\to C_b\big(\bbS_{\vartheta,T},\accentset{\,\,\,\circ}{\A}^{m,p}_{N;0,\C}\big)$,
is Lipschitz continuous.
\end{Th}

\begin{proof}[Proof of Theorem \ref{th:NS-complex}]
Assume that $m>2+\frac{d}{p}$, $N\in\Z_{\ge 0}$, and $1<p<\infty$.
The proof of the theorem will follow from Proposition \ref{prop:NS-modified} and a fixed point argument for the parabolic equation 
\eqref{eq:NS-modified'} that involves the Duhamel's formula, Theorem  \ref{th:SonA}, Corollary \ref{coro:S(t)-estimatesA}, 
Corollary \ref{coro:S(t)-estimatesA*}, and Lemma \ref{lem:lipschitz} (Remark \ref{rem:lipschitz}). 
Indeed, let us take $\rho, T_0>0$ and then choose $0<T<T_0$ and $\vartheta\in(0,\pi/2)$. 
For the simplicity of the exposition we will assume that $0<T_0<1$.
It follows from Corollary \ref{coro:S(t)-estimatesA} that there exists $M\equiv M(T_0,\nu)>0$ such that for 
any $v\in\A^{m,p}_{N;0,\C}$ and $z\in\bbS_{\vartheta,T}$,
\begin{equation}\label{eq:S(t)-estimate1}
\|S_\nu(z)v\|_{\A^{m,p}_{N;0,\C}}\le M\,\|v\|_{\A^{m,p}_{N;0,\C}}
\end{equation}
and the map $S_\nu(z) : \bbS_{\vartheta,T}\to\LL(\A^{m,p}_{N;0,\C})$ is holomorphic.
Similarly, it follows from Corollary \ref{coro:S(t)-estimatesA*} and the assumption $0<T_0<1$ that there exists a constant 
$C\equiv C(T_0,\nu)>0$ such that for any $v\in\A^{m-1,p}_{1,N+1;0,\C}$ and $z\in\bbS_{\vartheta,T}$,
\begin{equation}\label{eq:S(t)-estimate2}
\|S_\nu(z)v\|_{\A^{m,p}_{N;0,\C}}\le\frac{C}{\sqrt{|z|}}\,\|v\|_{\A^{m-1,p}_{1,N+1;0,\C}}
\end{equation}
and the map $S_\nu(z) : \bbS_{\vartheta,T}\to\LL(\A^{m-1,p}_{1,N+1;0,\C},\A^{m,p}_{N;0,\C})$ is holomorphic.
Now, take $u_0\in B_{\A^{m,p}_{N;0}}(\rho)$ and consider the {\em closed} subset
\[
{\mathcal B}_{\rho,T,u_0}^m\!\!:=\Big\{u : \bbS_{\vartheta,T}\to\A^{m,p}_{N;0,\C}\,\Big|\, u\,\,\text{is holomorphic},
\text{\rm continuous at zero}, u(0)=u_0, |u|_{\vartheta,T}\le M\rho+1\Big\}
\]
of the Banach space $C_b\big(\bbS_{\vartheta,T},\A^{m,p}_{N;0,\C}\big)$ equipped with the norm
$|u|_{\vartheta,T}:=\sup_{z\in\bbS_{\vartheta,T}}\|u(z)\|_{\A^{m,p}_{N;0,\C}}$ where $M>0$ is the constant appearing
in \eqref{eq:S(t)-estimate1}.
For any $u\in{\mathcal B}_{\rho,T,u_0}^m$ define the non-linear transform suggested by the Duhamel's principle
\begin{equation}\label{eq:T}
{\mathcal T}(u)(z):=S_\nu(z)u_0+\int_0^zS_\nu(z-\lambda)F\big(u(\lambda)\big)\,d\lambda,\quad z\in\bbS_{\vartheta,T},
\end{equation}
where the integration is taken over the (straight) segment $[0,z]\subseteq\bbS_{\vartheta,T}$.
It follows from \eqref{eq:S(t)-estimate1}, \eqref{eq:S(t)-estimate2}, Corollary \ref{coro:S(t)-estimatesA}, 
and Corollary \ref{coro:S(t)-estimatesA*}, that ${\mathcal T}(u) : \bbS_{\vartheta,T}\to\A^{m,p}_{N;0,\C}$ is 
continuous at zero bounded holomorphic map such that ${\mathcal T}(u)(0)=u_0$ (see Lemma \ref{lem:T-map} below). 
Moreover, by \eqref{eq:S(t)-estimate1}, \eqref{eq:S(t)-estimate2}, and Lemma \ref{lem:lipschitz} (Remark \ref{rem:lipschitz}),
for any given $u\in{\mathcal B}_{\rho,T,u_0}^m$ and $z\in\bbS_{\vartheta,T}$,
\begin{eqnarray}
\big\|{\mathcal T}(u)(z)\big\|_{\A^{m,p}_{N;0,\C}}&\le&\big\|S_\nu(z)u_0\big\|_{\A^{m,p}_{N;0,\C}}+
\int_0^z\big\|S_\nu(z-\lambda)F\big(u(\lambda)\big)\big\|_{\A^{m,p}_{N;0,\C}}\,|d\lambda|\nonumber\\
&\le&M\rho+C\Big(\int_0^z\frac{|d\lambda|}{\sqrt{|z-\lambda|}}\Big)\,
\sup_{\lambda\in\bbS_{\vartheta,T}}\big\|F(u(\lambda))\big\|_{\A^{m-1,p}_{1,N+1;0,\C}}\nonumber\\
&\le&M\rho+C\kappa\rho\Big(\int_0^z\frac{|d\lambda|}{\sqrt{|z-\lambda|}}\Big)\,
\sup_{\lambda\in\bbS_{\vartheta,T}}\|u(\lambda)\|_{\A^{m,p}_{N;0,\C}}\nonumber\\
&\le&M\rho+2C\kappa\rho\big(M\rho+1\big)\sqrt{|z|}\,.\label{eq:T-estimate1}
\end{eqnarray}
By choosing
\begin{equation}\label{eq:time_of_existence}
0<T<\min\Big\{T_0,1\big/4C^2\kappa^2\rho^2\big(M\rho+1\big)^2\Big\}
\end{equation} 
we then conclude from \eqref{eq:T-estimate1} that ${\mathcal T}(u)\in{\mathcal B}_{\rho,T,u_0}^m$. 
Hence, the transformation \eqref{eq:T} preserves the set ${\mathcal B}_{\rho,T,u_0}^m$,
defining this way a map 
\begin{equation}\label{eq:T-map}
{\mathcal T} : {\mathcal B}_{\rho,T,u_0}^m\to{\mathcal B}_{\rho,T,u_0}^m.
\end{equation}
We will show that \eqref{eq:T-map} is a contraction. In fact, by arguing in the same way as above 
and by using \eqref{eq:S(t)-estimate1}, \eqref{eq:S(t)-estimate2}, and Lemma \ref{lem:lipschitz} (Remark \ref{rem:lipschitz}), we obtain that
for any $u_1,u_2\in{\mathcal B}_{\rho,T,u_0}^m$ and for any $z\in\bbS_{\vartheta,T}$,
\begin{eqnarray}
\big\|{\mathcal T}(u_1)(z)-{\mathcal T}(u_2)(z)\big\|_{\A^{m,p}_{N;0,\C}}&\le& 
C\Big(\int_0^z\frac{|d\lambda|}{\sqrt{|z-\lambda|}}\Big)\,\sup_{\bbS_{\vartheta,T}}\big\|F(u_1)-F(u_2)\big\|_{\A^{m-1,p}_{1,N+1;0,\C}}\nonumber\\
&\le&C\kappa\rho\Big(\int_0^z\frac{|d\lambda|}{\sqrt{|z-\lambda|}}\Big)
\sup_{\bbS_{\vartheta,T}}\|u_1-u_2\|_{\A^{m,p}_{N;0,\C}}\nonumber\\
&\le&2C\kappa\rho\sqrt{|z|}\,\sup_{\bbS_{\vartheta,T}}\|u_1-u_2\|_{\A^{m,p}_{N;0,\C}}.\label{eq:T-lipschitz}
\end{eqnarray}
By the estimate \eqref{eq:time_of_existence}, the coefficient $\alpha:=2C\kappa\rho\sqrt{T}<1$. Hence, \eqref{eq:T-map} is
a contraction in ${\mathcal B}_{\rho,T,u_0}^m$. By the Banach contraction mapping theorem, there exists a unique $u\in{\mathcal B}_{\rho,T,u_0}^m$ 
so that $u={\mathcal T}(u)$. Hence, there exists a unique solution $u\in{\mathcal B}_{\rho,T,u_0}^m$ of the integral equation
\begin{equation}\label{eq:mild_solution}
u(z)=S_\nu(z)u_0+\int_0^z S_\nu(z-\lambda)F\big(u(\lambda)\big)\,d\lambda,\quad z\in\bbS_{\vartheta,T}.
\end{equation}
Now, take $u_0,v_0\in B_{\A^{m,p}_{N;0}}(\rho)$ and let $u\in{\mathcal B}_{\rho,T,u_0}^m$ and $v\in{\mathcal B}_{\rho,T,v_0}^m$ be
the solutions of \eqref{eq:mild_solution} so that $u(0)=u_0$ and $v(0)=v_0$. Then, by arguing as in the proof of \eqref{eq:T-lipschitz} we estimate
\[
\big\|u(z)-v(z)\big\|_{\A^{m,p}_{N;0,\C}}\le M\|u_0-v_0\|_{\A^{m,p}_{N;0,\C}}+\alpha\sup_{\bbS_{\vartheta,T}}\|u-v\|_{\A^{m,p}_{N;0,\C}},
\quad z\in\bbS_{\vartheta,T},
\]
where the constant $0<\alpha<1$ is chosen above. This implies that
\begin{equation}\label{eq:lipschitz_continuity}
\sup_{\bbS_{\vartheta,T}}\|u-v\|_{\A^{m,p}_{N;0,\C}}\le L\|u_0-v_0\|_{\A^{m,p}_{N;0,\C}},\quad\text{where}\quad L:=M/(1-\alpha).
\end{equation}
Therefore, the solution of \eqref{eq:mild_solution} depends Lipschitz continuously on $u_0\in B_{\A^{m,p}_{N;0}}(\rho)$ in 
the space $C_b\big(\bbS_{\vartheta,T},\A^{m,p}_{N;0,\C}\big)$.
This also implies that the solution of \eqref{eq:mild_solution} is unique not only in ${\mathcal B}_{\rho,T,u_0}^m$ but also
in $C_b\big(\bbS_{\vartheta,T},\A^{m,p}_{N;0,\C}\big)$.
It follows from Lemma \ref{lem:T-map} below that the (complex) derivative $u_z : \bbS_{\vartheta,T}\to\A^{m,p}_{N;0,\C}$ satisfies
\begin{equation}\label{eq:NS-complex}
u_z=\nu\Delta u+F(u),\quad u|_{z=0}=u_0\,.
\end{equation}
In particular, we see from \eqref{eq:NS-complex} that $u_z : \bbS_{\vartheta,T}\to\A^{m-2,p}_{N;0,\C}$
extends to a continuous map $\bbS_{\vartheta,T}\cup\{0\}\to\A^{m-2,p}_{N;0,\C}$.
By restricting the variable $z$ to the interval $[0,T)$ we then conclude that
$u|_{[0,T)}\in C\big([0,T),\A^{m,p}_{N;0}\big)\cap C^1\big([0,T),\A^{m-2,p} _{0,N}\big)$ 
is a solution of \eqref{eq:NS-modified'} (and \eqref{eq:NS-modified}).
By combining this with Proposition \ref{prop:NS-modified} we see that 
$u|_{[0,T)}\in C\big([0,T),\accentset{\,\,\,\circ}\A^{m,p}_{N;0}\big)\cap C^1\big([0,T),\accentset{\,\,\,\circ}\A^{m-2,p} _{0,N}\big)$ 
is a solution of the Navier-Stokes equation \eqref{eq:NS}. This completes the proof of Theorem \ref{th:NS-complex}.
\end{proof}

The following important technical lemma is used in the proof of Theorem \ref{th:NS}.

\begin{Lem}\label{lem:T-map}
Assume that $m>2+\frac{d}{p}$, $N\in\Z_{\ge 0}$, and $1<p<\infty$.
Take $\vartheta\in(0,\pi/2)$ and $0<T<1$ and let $u :\bbS_{\vartheta,T}\to\A^{m,p}_{N;0,\C}$ be 
a continuous at zero bounded holomorphic map such that $u(0)=u_0$. Then, the transformation \eqref{eq:T} gives 
a holomorphic map $\mathcal{T}(u) : \bbS_{\vartheta,T}\to\A^{m,p}_{N;0,\C}$ that is bounded, continuous at zero, 
and $\mathcal{T}(u)(0)=u_0$. On $\bbS_{\vartheta,T}$ one has $\mathcal{T}(u)_z=\nu\Delta\mathcal{T}(u)+F(u)$ where $F$ is 
the map \eqref{eq:F-complexified}.
\end{Lem}

\begin{Rem}\label{rem:T-map}
In fact, the proof of Lemma \ref{lem:T-map} bellow shows that if $u :\bbS_{\vartheta,T}\to\A^{m,p}_{N;0,\C}$ depends on 
a parameter $\eta\in\U$, where $\U$ is an open set in a (complex) Banach space, so that
$u : \bbS_{\vartheta,T}\times \U\to\A^{m,p}_{N;0,\C}$, $(z,\eta)\mapsto  u(z;\eta)$, is analytic
then $\mathcal{T}(u) : \bbS_{\vartheta,T}\times\U\to\A^{m,p}_{N;0,\C}$ is analytic.
\end{Rem}

\begin{proof}[Proof of Lemma \ref{lem:T-map}]
Assume that $\|u_0\|_{\A^{m,p}_{N;0}}<\rho$ and $\sup_{z\in\bbS_{\vartheta,T}}\|u(z)\|_{\A^{m,p}_{N;0}}<\rho_1$ 
for some $\rho,\rho_1>0$. 
Recall from the proof of Theorem \ref{th:NS} that the inequalities \eqref{eq:S(t)-estimate1} and \eqref{eq:S(t)-estimate2} hold
on $\bbS_{\vartheta,T}$. The first term on the right side of \eqref{eq:T} is  holomorphic and continuous at zero as required 
since by Theorem \ref{th:SonA}, $\{S_\nu(t)\}_{t\ge 0}$ is an analytic semigroup on $\A^{m,p}_{N;0,\C}$ with angle $\pi/2$.
Let us now concentrate our attention to the second term in \eqref{eq:T},
\begin{equation}\label{eq:w}
w(z):=\int_0^zS_\nu(z-\lambda)F\big(u(\lambda)\big)\,d\lambda
=\int_0^zS_\nu(s)F\big(u(z-s)\big)\,ds,\quad z\in\bbS_{\vartheta,T},
\end{equation}
where the integration is taken over the (straight) segment $[0,z]\subseteq\bbS_{\vartheta,T}$.
It follows from \eqref{eq:F-complexified}, \eqref{eq:S(t)-estimate2}, and Corollary \ref{coro:S(t)-estimatesA*} (applied with $\tau=1$), 
that the first integral in \eqref{eq:w} is well defined as an improper Riemann integral in $\A^{m,p}_{N;0,\C}$ with 
an integrable singularity at $\lambda=z$. Similarly, the second integral in \eqref{eq:w} is well defined as an improper Riemann integral 
in $\A^{m,p}_{N;0,\C}$ with an integrable singularity at $s=0$.
Let us now fix $z_0\in\bbS_{\vartheta,T}$ and choose an open disk $U(z_0)$ in $\bbS_{\vartheta,T}$
centered at $z_0$. Then we choose $z,\tz\in\C$ such that
\begin{equation}\label{eq:z-condition}
z\in U(z_0)\subseteqq\bbS_{\vartheta,T}
\end{equation}
and
\begin{equation}\label{eq:tz-condition}
\tz\in\bbS_{\vartheta,T}\quad\text{\rm and}\quad z-\tz\in\bbS_{\vartheta,T},
\end{equation}
as well as a sequence of complex numbers $(z_k)_{k\ge 1}$ such that $z_k\to 0$ as $k\to\infty$ and
\begin{equation}\label{eq:z_k-condition}
z_k\in\bbS_{\vartheta,T}\quad\text{\rm and}\quad z-z_k\in\bbS_{\vartheta,T}\,\,\forall z\in U(z_0)
\end{equation}
for any $k\ge 1$. Note that conditions \eqref{eq:z-condition} and \eqref{eq:tz-condition} define an open $W$ set of pairs $(z,\tz)$ in $\C^2$.
The second condition in \eqref{eq:z_k-condition} can be satisfied by taking $z_k$, $k\ge 1$, sufficiently close to zero.
For $(z,\tz)\in W$ and $k\ge 1$ consider the auxiliary integral
\begin{equation}\label{eq:w_k}
w_k(z,\tz):=\int_{z_k}^\tz S_\nu(s)F\big(u(z-s)\big)\,ds
\end{equation}
where the integration is taken over the segment $[z_k,\tz]\subseteq\bbS_{\vartheta,T}$.  
It follows from \eqref{eq:z-condition}, \eqref{eq:tz-condition}, and \eqref{eq:z_k-condition}, that \eqref{eq:w_k} is well defined.
Moreover, it follows from \eqref{eq:F-complexified}, \eqref{eq:S(t)-estimate2}, and Corollary \ref{coro:S(t)-estimatesA*}, that 
for any $k\ge 1$, 
\begin{equation}\label{eq:w_k-map_local}
w_k : W\to\A^{m,p}_{N;0,\C},\quad (z,\tz)\mapsto w_k(z,\tz),
\end{equation}
is a holomorphic map for any choice of $z_0\in\bbS_{\vartheta,T}$.
(Note that the integrand in \eqref{eq:w_k} has no singularities in an open neighborhood of $[z_k,\tz]$ and depends 
(locally) holomorphically on $z$.)
It follows from \eqref{eq:w_k}, \eqref{eq:S(t)-estimate2}, Lemma \ref{lem:lipschitz} (Remark \ref{rem:lipschitz}), 
and the assumptions on $u$, that for any $(z,\tz)\in W$ and for any $k,l\ge 1$,
\begin{eqnarray}\label{eq:w_k-Cauchy_sequence}
\big\|w_k(z,\tz)-w_l(z,\tz)\big\|_{\A^{m,p}_{N;0,\C}}&=&\Big\|\int_{z_k}^{z_l}S_\nu(s)F\big(u(z-s)\big)\,ds\Big\|_{\A^{m,p}_{N;0,\C}}\nonumber\\
&\le&C\Big(\int_{z_k}^{z_l}\frac{|ds|}{\sqrt{|s|}}\Big)\sup_{s\in\bbS_{\vartheta,T}}\|F\big(u(s)\big)\|_{\A^{m-1,p}_{1,N+1;0,\C}}\nonumber\\
&\le&C\kappa\rho_1\Big(\int_{z_k}^{z_l}\frac{|ds|}{\sqrt{|s|}}\Big)\sup_{s\in\bbS_{\vartheta,T}}\|u(s)\|_{A^{m,p}_{N;0,\C}}\nonumber\\
&\le&\frac{2\,C\kappa\rho_1^2}{\cos\vartheta}\sqrt{|z_k-z_l|}
\end{eqnarray}
where the integration is taken over the segment $[z_k,z_l]$.
In \eqref{eq:w_k-Cauchy_sequence} we use that for $z_k\ne z_l$,
\begin{equation}\label{eq:integral_inequality}
\int_{z_k}^{z_l}\frac{|ds|}{\sqrt{|s|}}\le\int_0^1\frac{|z_k-z_l|\,dt}{\sqrt{t|z_k-z_l|}}\le
\frac{2\,\sqrt{|z_k-z_l|}}{\cos\vartheta},
\end{equation}
which follows from the fact that for any $0\le t\le 1$,
\begin{equation}\label{eq:integral_inequality'}
\big|z_k+t(z_l-z_k)\big|\ge\big|z_k'+t(z_l'-z_k')\big|\ge t\,|z_l'-z_k'|\ge t\,|z_l-z_k|\cos\vartheta,
\end{equation}
where $z_k'$ and $z_l'$ are the orthogonal projections of $z_k$ and $z_l$ onto a lines $\ell_\theta$ in $\C$ 
that passes though zero, makes an angle $\vartheta$ with the $x$-axis, and is chosen so that 
the length of the orthogonal projection of the interval $[z_k,z_l]$ onto $\ell_\theta$ is greater or equal than 
the length of the orthogonal projection of the interval $[z_k,z_l]$ onto the other line that passes through zero and 
makes angle $\vartheta$ with the $x$ axis. (One easily sees that $|z_l'-z_k'|\ge |z_l-z_k|\cos\vartheta>0$
where the equality happens when $[z_k,z_l]$ is parallel to the $x$-axis. Since \eqref{eq:integral_inequality} is 
independent of the order of $z_k'$ and $z_l'$, in \eqref{eq:integral_inequality'} we assume without loss of generality 
that $|z_k'|<|z_l'|$.) It now follows from \eqref{eq:w_k-Cauchy_sequence} that $w_k$, $k\ge 1$, converges 
in $\A^{m,p}_{N;0,\C}$ to $\widetilde{w} : W\to\A^{m,p}_{N;0,\C}$ uniformly in $W$. 
Since the limit function is independent of the choice of the sequence $(z_k)_{k\ge 1}$ described above, 
we conclude from \eqref{eq:w_k} by taking the points $z_k$, $k\ge 1$, on $[0,z]$ that 
\begin{equation}\label{eq:tw}
\widetilde{w}(z,\tz)=\int_0^{\tz}S_\nu(s)F\big(u(z-s)\big)\,ds
\end{equation}
for any $(z,\tz)\in W$. By recalling that for any $k\ge 1$ the map \eqref{eq:w_k-map_local} is holomorphic, 
we obtain from the uniform convergence that 
\begin{equation}\label{eq:tw-map_local}
\widetilde{w} : W\to\A^{m,p}_{N;0,\C},\quad(z,\tz)\mapsto w(z,\tz),
\end{equation}
is holomorphic for any choice of $z_0\in\bbS_{\vartheta,T}$. 
For $(z,\tz_1)$ and $(z,\tz_2)$ in $W$ taken so that $|\tz_1|, |\tz_2|>\delta$ where
$\delta>0$ is the half of the distance between zero and $U(z_0)$ in $\C$, 
we obtain from \eqref{eq:tw}, by arguing as in \eqref{eq:w_k-Cauchy_sequence}, that
\begin{eqnarray}\label{eq:tw-Cauchy_sequence}
\big\|\widetilde{w}(z,\tz_2)-\widetilde{w}(z,\tz_1)\big\|_{\A^{m,p}_{N;0,\C}}&=&
\Big\|\int_{\tz_1}^{\tz_2}S_\nu(s)F\big(u(z-s)\big)\,ds\Big\|_{\A^{m,p}_{N;0,\C}}\nonumber\\
&\le&C\Big(\int_{\tz_1}^{\tz_2}\frac{|ds|}{\sqrt{|s|}}\Big)\sup_{s\in\bbS_{\vartheta,T}}
\|F\big(u(s)\big)\|_{\A^{m-1,p}_{1,N+1;0,\C}}\nonumber\\
&\le&\frac{C\kappa\rho_1^2}{\sqrt{\delta}}\,|\tz_2-\tz_1|
\end{eqnarray}
where the integration is taken over the segment $[\tz_1,\tz_2]\subseteq\bbS_{\vartheta,T}$.
Now, take $z\in U(z_0)$ and define
\begin{equation}\label{eq:tw_k}
\widetilde{w}_k(z):=\widetilde{w}\big(z,\tz_k(z)\big)\quad\text{\rm where}\quad\tz_k(z):=\Big(1-\frac{1}{2k}\Big)z
\end{equation}
for any $k\ge 1$. Note that $(z,\tz_k(z))\in W$ for any $k\ge 1$. 
This and the analyticity of \eqref{eq:tw-map_local} imply that the map
\begin{equation}\label{eq:tw_k-map-local}
\widetilde{w}_k : U(z_0)\to\A^{m,p}_{N;0,\C}
\end{equation}
is holomorphic. It follows from \eqref{eq:tw-Cauchy_sequence} that $\widetilde{w}_k$ converges in $\A^{m,p}_{N;0,\C}$ to
\begin{equation}\label{eq:w_k-copy}
w(z)=\int_0^zS_\nu(s)F\big(u(z-s)\big)\,ds
\end{equation}
uniformly in $z\in U(z_0)$.
Hence, the map $w : U(z_0)\to\A^{m,p}_{N;0,\C}$ is holomorphic for any choice of $z_0\in\bbS_{\vartheta,T}$. This shows that
\begin{equation}\label{eq:w-map}
w : \bbS_{\vartheta,T}\to\A^{m,p}_{N;0,\C}
\end{equation}
is holomorphic. Moreover, the first representation of $w$ in \eqref{eq:w} and estimates as in \eqref{eq:T-estimate1} show that 
\[
\|w(z)\|_{\A^{m,p}_{N;0,\C}}=O\big(\sqrt{|z|}\big)
\] 
for $z\in\bbS_{\vartheta,T}$, which implies that \eqref{eq:w-map} is bounded and continuous at zero.
This completes the proof of the first statement of the lemma.

Let us now prove the last statement of the lemma. 
Note that $\mathcal{T}(u)|_{(0,T)}=\mathcal{T}_{\rm real}\big(u|_{(0,T)}\big)$
where $\mathcal{T}_{\rm real}$ is given by \eqref{eq:T-map_real} in Appendix \ref{sec:aux-results}.
It then follows from Lemma \ref{lem:T-map_real} in Appendix \ref{sec:aux-results} that 
\[
\mathcal{T}(u)|_{(0,T)}\in C\big([0,T),\A^{m,p}_{N;0}\big)\cap C^1\big((0,T),\A^{m-2,p}_{N;0}\big)
\]
satisfies
\begin{equation}\label{eq:T-equation}
\mathcal{T}(u)_t=\nu\Delta\mathcal{T}(u)+F(u),\quad\mathcal{T}(u)(0)=u_0\,.
\end{equation}
Since $u, \mathcal{T}(u) : \bbS_{\vartheta,T}\to\A^{m,p}_{N;0,\C}$ are real-analytic holomorphic maps,
$\mathcal{T}(u)_z$ and $\nu\Delta\mathcal{T}(u)+F(u)$ are real-analytic holomorphic maps 
$\bbS_{\vartheta,T}\to\A^{m-2,p}_{N;0,\C}$ (cf. Remark \ref{rem:lipschitz}). 
By combining this with \eqref{eq:T-equation} we then obtain that
\[
\mathcal{T}(u)_z=\nu\Delta\mathcal{T}(u)+F(u)
\]
on $\bbS_{\vartheta,T}$. This completes the proof of the lemma.
\end{proof}

\begin{Rem}\label{rem:analytic_dependence}
Let $u(z;u_0)$, $z\in\bbS_{\vartheta,T}$, $u_0\in B^{m,p}_{N;0}(\rho)\cap\accentset{\,\,\,\circ}{\A}^{m,p}_{N;0}$,
be the holomorphic solution in Theorem \ref{th:NS-complex}. Then, the map
$\bbS_{\vartheta,T}\times\Big(B^{m,p}_{N;0}(\rho)\cap\accentset{\,\,\,\circ}{\A}^{m,p}_{N;0}\Big)\to
\accentset{\,\,\,\circ}{\A}^{m,p}_{N;0}$, $(z,u_0)\mapsto u(z;u_0)$,
is analytic. In order to see this, we choose $u_0$ for the initial guess in the iterative process in the proof of Theorem \ref{th:NS-complex}
and then use Remark \ref{rem:T-map} to conclude that for any $k\ge 1$ the $k$-th iterate of the map \eqref{eq:T},
\[
\bbS_{\vartheta,T}\times\Big(B^{m,p}_{N;0}(\rho)\cap\accentset{\,\,\,\circ}{\A}^{m,p}_{N;0}\Big)\to
\accentset{\,\,\,\circ}{\A}^{m,p}_{N;0},\quad(z,u_0)\mapsto\mathcal{T}^k(u_0)(z), 
\]
is analytic. The claimed analyticity of the solution $u(z;u_0)$ then follows since the estimates \eqref{eq:T-estimate1} and \eqref{eq:T-lipschitz}
in the proof of Theorem \ref{th:NS-complex} can be made locally uniform in $\eta\in\U$ (This follows easily by inspection.)
\end{Rem}

Let us now prove Theorem \ref{th:NS}.

\begin{proof}[Proof of Theorem \ref{th:NS}]
We follow the notation in the proof of Theorem \ref{th:NS-complex}.
Take $\rho>0$, $\vartheta\in(0,\pi/2)$.
Then, by Theorem \ref{th:NS-complex}, there exists $T'\equiv T'(\rho,\vartheta)>0$ such that for any divergence free
$u_0\in B^{m,p}_{N;0}(\rho)$ there exists a holomorphic solution 
$u\in C_b\big(\bbS_{\vartheta,T'},\accentset{\,\,\,\circ}\A^{m,p}_{N;0,\C}\big)$ of the Navier-Stokes equation on $\bbS_{\vartheta,T'}$ 
that depends Lipschitz continuously on the initial data $u_0\in B^{m,p}_{N;0}(\rho)\cap\accentset{\,\,\,\circ}\A^{m,p}_{N;0}$.
Now, take $0<T<T'$. By the definition of a holomorphic solution (cf. Definition \ref{def:NS-holomorphic_solution}), we have that
\begin{equation}\label{eq:u-restricted}
u|_{[0,T]}\in C\big([0,T],\accentset{\,\,\,\circ}\A^{m,p}_{N;0}\big)\cap C^1\big([0,T],\accentset{\,\,\,\circ}\A^{m-2,p}_{N;0}\big),
\end{equation}
is a solution of the Navier-Stokes equation \eqref{eq:NS}, and hence it satisfies \eqref{eq:NS-modified'} (Proposition \ref{prop:NS-modified}).

Let us now prove the uniqueness of solutions of the Navier-Stokes equation \eqref{eq:NS} in the class 
$C\big([0,T],\accentset{\,\,\,\circ}\A^{m,p}_{N;0}\big)\cap C^1\big([0,T],\accentset{\,\,\,\circ}\A^{m-2,p}_{N;0}\big)$.
(Note that such uniqueness does not follow from the Lipschitz continuity on the initial data of the solutions $u|_{[0,T]}$, where
$u\in C_b\big(\bbS_{\vartheta,T'},\accentset{\,\,\,\circ}\A^{m,p}_{N;0,\C}\big)$ is the holomorphic solution constructed above, 
since at this point we do  not know if any solution of the Navier-Stokes equation in 
$C\big([0,T],\accentset{\,\,\,\circ}\A^{m,p}_{N;0}\big)\cap C^1\big([0,T],\accentset{\,\,\,\circ}\A^{m-2,p}_{N;0}\big)$
extends to a holomorphic one.) To this end, take
\[
u\in C\big([0,T],\accentset{\,\,\,\circ}\A^{m,p}_{N;0}\big)\cap C^1\big([0,T],\accentset{\,\,\,\circ}\A^{m-2,p}_{N;0}\big)
\]
to be a solution of the Navier-Stokes equation such that $u|_{t=0}=u_0$.
Then, by Proposition \ref{prop:NS-modified}, such $u$ satisfies equation \eqref{eq:NS-modified'} (and \eqref{eq:NS-modified}).
Hence, $u$ satisfies the inhomogeneous heat equation on $[0,T]$,
\begin{equation}\label{eq:inhomogeneous_heat_equation}
u_t=\nu\Delta u+f,\quad u|_{t=0}=u_0,
\end{equation}
where $f(t):=F\big(u(t)\big)$ and $f : [0,T]\to\A^{m-2,p}_{N;0}$ is continuous by Lemma \ref{lem:lipschitz}.
By Lemma \ref{lem:T-map_real}, 
\[
\mathcal{T}_{\rm real}(u)\in C\big([0,T),\A^{m,p}_{N;0}\big)\cap C^1\big([0,T),\A^{m-2,p}_{N;0}\big)
\]
also satisfies \eqref{eq:inhomogeneous_heat_equation}. In view of the uniqueness in Proposition \ref{prop:heat_equation_in_S'}
we then conclude that $u=\mathcal{T}_{\rm real}(u)$ on $[0,T)$. Hence, $u$ satisfies 
\begin{equation}\label{eq:mild_solution(real)}
u(t)=S_\nu(t)u_0+\int_0^t S_\nu(t-s)F\big(u(s)\big)\,ds,\quad t\in[0,T].
\end{equation}
Now, we take $u,v\in C\big([0,T],\accentset{\,\,\,\circ}\A^{m,p}_{N;0}\big)\cap C^1\big([0,T],\accentset{\,\,\,\circ}\A^{m-2,p}_{N;0}\big)$
to be solutions of the Navier-Stokes equation so that $u_{t=0}=u_0$ and $v_{t=0}=v_0$ with 
$u_0,v_0\in B_{\A^{m,p}_{N;0}}(\rho)\cap\accentset{\,\,\,\circ}\A^{m,p}_{N;0}$. 
Then, since they both satisfy \eqref{eq:mild_solution(real)}, we can argue exactly as in the proof of \eqref{eq:lipschitz_continuity} above 
to conclude that 
\[
\sup_{[0,T]}\|u-v\|_{A^{m,p}_{N;0}}\le L\|u_0-v_0\|_{\A^{m,p}_{N;0}}
\] 
holds with a constant  $L>0$ depending on the choice of $\rho>0$. This together with \eqref{eq:NS-modified'}
and Lemma \ref{lem:lipschitz} proves that the solutions of the Navier-Stokes equation 
in $C\big([0,T],\accentset{\,\,\,\circ}\A^{m,p}_{N;0}\cap C^1\big([0,T],\accentset{\,\,\,\circ}\A^{m-2,p}_{N;0}\big)$ depend
Lipschitz continuously on bounded sets of initial data in $\accentset{\,\,\,\circ}\A^{m,p}_{N;0}$. In particular the solutions are unique.
\end{proof}

A direct consequence of Theorem \ref{th:NS} is the following

\begin{Coro}\label{coro:NS-global}
We have
\begin{itemize}
\item[(i)] For any $u_0\in\accentset{\,\,\,\circ}\A^{m,s}_{N;0}$ there exist a maximal time of existence $T_\infty>0$ and a unique solution 
$u\in C\big([0,T_\infty),\accentset{\,\,\,\circ}\A^{m,p}_{N;0}\big)\cap C^1\big([0,T_\infty),\accentset{\,\,\,\circ}\A^{m-2,p}_{N;0}\big)$ 
of the Navier-Stokes equation \eqref{eq:NS}. If $T_\infty<\infty$ then the $\A^{m,p}_{N;0}$-norm of $u$ blows up, i.e.,
$\lim_{t\to T_\infty-0}\|u(t)\|_{A^{m,p}_{N;0}}=\infty$.
\item[(ii)] The solution $u|_{(0,T_\infty)} : (0,T_\infty)\to\accentset{\,\,\,\circ}\A^{m,p}_{N;0}$ is real-analytic.
\end{itemize}
\end{Coro}

Item $(i)$ can be proved in a standard way and will be thus omitted. 
Item $(ii)$ follows directly from Theorem \ref{th:NS-complex}.

\appendix
\section{Properties of asymptotic spaces}\label{sec:appendix_properties}
Let us summarize some of the essential properties of the weighted and asymptotic spaces that were defined in Section \ref{sec:introduction}. 
For further details, see \cite[Appendix B]{McOwenTopalov2} and \cite[Appendix C]{McOwenTopalov4}. 
Assume that $1<p<\infty$ and recall from \cite[Lemma 2.2]{McOwenTopalov2} that for any $m\ge 0$, $\delta\in\R$, $1\le k\le d$, 
$\partial_k : W^{m+1,p}_\delta\to W^{m,p}_{\delta+1}$ is bounded. 

\begin{Prop}\label{pr:A-properties}  
Assume that $m\ge 0$. Then, we have
\begin{enumerate}
\item[(i)] If $n_1\geq n$, $N_1\geq N$, and $\ell_1\leq\ell$, then we have a continuous inclusion 
${\mathcal A}_{n_1,N_1;\ell_1}^{m,p}\subseteq {\mathcal A}_{n,N;\ell}^{m,p}$.
\item[(ii)]  If $m\geq 1$,  then $u\mapsto \partial u/\partial x_j$ is a bounded linear map 
${\mathcal A}_{n,N;\ell}^{m,p}\to{\mathcal A}_{n+1,N+1;\ell-1}^{m-1,p}$.
\item[(iii)] Multiplication by $\chi(r)\,r^{-k}$ is bounded ${\mathcal A}_{n,N;\ell}^{m,p}\!\to\! {\mathcal A}_{n+k,N+k;\ell-k}^{m,p}$.
\item[(iv)] Multiplication by $\chi(r)\,(\log r)^{j}$ is bounded ${\mathcal A}_{n,N;\ell}^{m,p}\!\to\! {\mathcal A}_{n,N^-;\ell+j}^{m,p}$
for any $N^-<N$.
\end{enumerate}
\end{Prop}

\noindent Next we consider products of functions in asymptotic spaces. 
Recall from \cite[Proposition 2.2, Lemma 2.2]{McOwenTopalov2} that, for $m>d/p$ and any $\delta_1,\delta_2\in\R$, 
pointwise multiplication of functions $(f,g)\mapsto fg$ defines a continuous map
\begin{equation*}
W^{m,p}_{\delta_1} \times W^{m,p}_{\delta_2} \to W^{m,p}_{\delta_1+\delta_2+\frac{d}{p}}.
\end{equation*}
This property is needed for the proof of the following propositions (see e.g. \cite[Appendix B]{McOwenTopalov2}).

\begin{Prop}\label{pr:A-products} 
Suppose $m>d/p$,  $0\leq n_i\leq N_i$,  and $\ell_i+n_i\geq 0$  for $i=1,2$. Let $n_0=n_1+n_2$ and $\ell_0=\ell_1+\ell_2$.
Then, we have
\begin{enumerate} 
\item[(i)] For any $N_0<\min(N_1+n_2,N_2+n_1)$ we have
\begin{equation}\label{est:A-multiplication1}
\|u\,v\|_{{\mathcal A}_{n_0,N_0;\ell_0}^{m,p} }\leq C\,\|u\|_{{\mathcal A}_{n_1,N_1;\ell_1}^{m,p} }\|v\|_{{\mathcal A}_{n_2,N_2;\ell_2}^{m,p}}
\quad\hbox{for}\ u\in {\mathcal A}_{n_1,N_1;\ell_1}^{m,p},\ v\in {\mathcal A}_{n_2,N_2;\ell_2}^{m,p}.
\end{equation}
\item[(ii)]  For $N_0=\min(N_1+n_2,N_2+n_1)$ we have 
\begin{equation}\label{est:A-multiplication2}
\|u\,v\|_{{\mathcal A}_{n_0,N_0;-n_0}^{m,p} }\leq C\,\|u\|_{{\mathcal A}_{n_1,N_1;-n_1}^{m,p} }\|v\|_{{\mathcal A}_{n_2,N_2;-n_2}^{m,p} }
\ \hbox{for}\ u\in {\mathcal A}_{n_1,N_1;-n_1}^{m,p},\ v\in {\mathcal A}_{n_2,N_2;-n_2}^{m,p}.
\end{equation}
\end{enumerate}
\end{Prop}

\noindent As a corollary of this proposition we find that the asymptotic space $\A^{m,p}_{n,N;0}$ forms a Banach algebra 
under pointwise multiplication:

\begin{Coro}\label{co:BanachAlgebras} 
If $m>d/p$ and $-N\le-n\le\ell\le 0$, then $\A^{m,p}_{n,N;\ell}$ is a Banach algebra. In particular, $\A^{m,p}_{n,N;0}$ is a Banach algebra
if $m>d/p$ and $0\le n\le N$.
\end{Coro}

\noindent The following two lemmas follows easily from the definition of the asymptotic spaces.

\begin{Lem}\label{lem:equivalent_W-norms}
Assume that $m\ge 0$ and $\delta\in\R$. Then $u\in W^{m+1,p}_\delta$ if and only if
$u\in W^{m,p}_\delta$ and $\partial_k u\in W^{m,p}_{\delta+1}$. Moreover,
the norm in $W^{m+1,p}_\delta$ is equivalent to the norm
\[
|u|_{W^{m+1,p}_\delta}:=\|u\|_{W^{m,p}_\delta}+
\sum_{k=1}^d\|\partial_k u\|_{W^{m,p}_{\delta+1}},\quad u\in W^{m+1,p}_\delta.
\]
\end{Lem}

\begin{Lem}\label{lem:equivalent_A-norms}
Assume that $m\ge 0$ and  $0\le n\le N$ are integer numbers. Then $u\in\A^{m+1,p}_{n,N;0}$ if and only if
$u\in\A^{m,p}_{n,N;0}$ and $\partial_k u\in\A^{m,p}_{n+1,N+1;-1}$. Moreover,
the norm in $\A^{m+1,p}_{n,N;0}$ is equivalent to the norm
\[
|u|_{\A^{m+1,p}_{n,N;0}}:=\|u\|_{\A^{m,p}_{n,N;0}}+
\sum_{k=1}^d\|\partial_k u\|_{\A^{m,p}_{n+1,N+1;-1}},\quad u\in\A^{m+1,p}_{n,N;0}.
\]
\end{Lem}

\noindent We will also need the following fact about the $W$-spaces.

\begin{Lem}\label{lem:properties_W}
Assume that $1<p<\infty$ and that $m\ge 0$ is an integer. Then we have:
\begin{itemize}
\item[(i)] The map
\[
H^{m,p}(\s^{d-1})\to W^{m,p}_{\gamma_N},\quad a(\theta)\mapsto a(\theta)/r^{N+1},
\]
is bounded.
\item[(ii)] Assume that $A\in C_c^\infty$ with support in the interval $[1,2]$. Then,
for any given weight $\delta\in\R$ we have that the map
\[
H^{m,p}(\s^{d-1})\to W^{m,p}_\delta,\quad a(\theta)\mapsto A(r)a(\theta),
\]
is bounded.
\end{itemize}
\end{Lem}

\noindent The proof of this lemma is straightforward and hence omitted.

\begin{Rem}\label{rem:complexified_spaces}
Note that the statements of this Appendix hold without changes for the corresponding complexified spaces.
\end{Rem}

\section{Auxiliary results}\label{sec:aux-results}
In this Appendix we collect several auxiliary results used in the main body of the paper.
We first prove the following form of Peetre's inequality.

\begin{Lem}\label{lem:elementary}
For any $\delta\in\R$ there is a constant $C\equiv C_\delta>0$ such that
\begin{equation}\label{eq:elementary}
\x^\delta/\y^\delta\le C\langle x-y\rangle^{|\delta|}
\end{equation}
for any $x,y\in\R^d$. Moreover, for any given $\varepsilon_0>0$ one has
\begin{equation}\label{eq:elementary*}
\x^\delta/\y^\delta=1+O\big(|x-y|\big)
\end{equation}
with constant depending on $\delta>0$ and $\varepsilon_0>0$ but uniform in 
$x,y\in\R^d$ with $|x-y|\le\varepsilon_0$.
\end{Lem}

\begin{proof}[Proof of Lemma \ref{lem:elementary}]
By the triangle inequality, we have 
$1+|y+z|^2\le 1+2(|y|^2+|z|^2)\le(1+|y|^2)(1+|z|^2)$ for all $y,z\in\R^d$. 
If $\delta>0$, this implies
\begin{equation}\label{eq:peetre}
\langle y+z\rangle^\delta\le 2^{\delta/2}\y^\delta \z^\delta\quad\forall y,z\in\R^d
\end{equation} 
and we let $z=x-y$ to obtain \eqref{eq:elementary}.
In view of \eqref{eq:peetre} we then obtain
$\langle y\rangle^\delta=\big\langle (y-x)+x\big\rangle^\delta\le 
2^{\delta/2}\langle y-x\rangle^\delta \x^\delta$ that proves
\eqref{eq:elementary} with $\delta<0$. 
The case when $\delta=0$ is trivial. Let us now prove the second statement of the lemma.
Take $\varepsilon_0>0$ and assume that $x,y\in\R^d$ satisfy $|x-y|\le\varepsilon_0$.
Then we have
\begin{eqnarray}\label{eq:delta=2}
\frac{\x^2}{\y^2}-1&=&\frac{\big|y+(x-y)\big|^2-|y|^2}{1+|y|^2}=
\Big(x-y,\frac{2y}{1+|y|^2}\Big)+\frac{|x-y|^2}{1+|y|^2}\nonumber\\
&\le&\frac{2|y|}{1+|y|^2}|x-y|+|x-y|^2\le(1+\varepsilon_0)|x-y|.
\end{eqnarray}
This proves the last statement of the lemma for $\delta=2$. In order to prove the general case we first note that
for any $r_0>0$ there exists $C\equiv C_{\delta,r_0}>0$ such that
$(1+r)^{\delta/2}\le 1+C r$ for any $|r|\le r_0$. By combining this with \eqref{eq:delta=2}
we conclude the proof of \eqref{eq:elementary*}.
\end{proof}

\medskip\medskip

Let $\Sz'$ be the space of tempered distributions. By definition, a curve $w : (0,T)\to\Sz'$, $T>0$, belongs to the class
$C^k\big((0,T),\Sz'\big)$ for integer $k\ge 0$ if for any test function $\varphi\in\Sz$ the map
\[
(0,T)\to\R,\quad t\mapsto\big\langle w(t),\varphi\big\rangle,
\]
where $\langle\cdot,\cdot\rangle$ denotes the pairing between $\Sz'$ and $\Sz$, 
belongs to $C^k\big((0,T),\R\big)$. One easily sees that $w\in C^k\big((0,T),\Sz'\big)$ implies that
there exist $w'\equiv w^{(1)},...,w^{(k)}\in C\big((0,T),\Sz'\big)$ such that for any $\varphi\in\Sz$,
\[
\langle w^{(j)}(t),\varphi\rangle=\frac{d^j}{dt^j}\big\langle w(t),\varphi\big\rangle,\quad 0\le j\le k\,.
\]
In a similar way one treats the case when the interval $(0,T)$ is replaced by $[0,T)$ or $[0,T]$.
We have the following

\begin{Lem}\label{le:uniqueness} 
Suppose that $w\in C\big([0,T),{\mathcal S}'\big)\cap C^1\big((0,T),{\mathcal S}'\big)$ satisfies $w'(t)=\nu\Delta w(t)$ as
tempered distributions for $t\in(0,T)$ and $w(0)=0$. Then $w(t)\equiv 0$ for all $t\in[0,T)$.
\end{Lem}

\begin{proof}[Proof of Lemma \ref{le:uniqueness}]
We apply Holmgren's principle. Take $t_0\in(0,T)$ and for any given $\varphi\in\Sz$ consider the curve
\[
(0,t_0)\to\Sz,\quad t\mapsto v(t):=G_{t_0-t}*\varphi,
\]
where $G_t(x)=\frac{1}{(4\nu\pi t)^{d/2}}e^{-\frac{|x|^2}{4\nu t}}$. Then, 
\[
v\in C\big([0,t_0),\Sz\big)\cap C^1\big((0,T),\Sz\big)
\]
and for any $t\in(0,t_0)$,
\begin{equation}\label{eq:heat_equation_v}
v'(t)=-\nu\Delta v(t),\quad v(t_0)=\varphi,
\end{equation}
where $v'(t)$ denotes the derivative of $v$ at time $t$ in the Frech\'et space $\Sz$.
Now, consider the curve $(0,t_0)\to\R$, $t\mapsto\big\langle w(t),v(t)\big\rangle$. For any $t\in(0,t_0)$ we obtain
from \eqref{eq:heat_equation_v} that
\begin{eqnarray}
\frac{d}{dt}\big\langle w(t),v(t)\big\rangle&=&\big\langle w'(t),v(t)\big\rangle+\big\langle w(t),v'(t)\big\rangle\label{eq:product_formula}\\
&=&\nu\big\langle\Delta w(t),v(t)\big\rangle-\nu\big\langle w(t),\Delta v(t)\big\rangle=0\label{eq:zero_derivative}.
\end{eqnarray}
The product formula \eqref{eq:product_formula} follows easily from Lemma \ref{lem:pairing} below.
Since $w\in C\big([0,t_0],\Sz'\big)$ and $v\in C\big([0,t_0],\Sz\big)$ we obtain from Lemma \ref{lem:pairing}
and \eqref{eq:heat_equation_v} that
\begin{equation}\label{eq:limits}
\lim_{t\to 0+}\big\langle w(t),v(t)\big\rangle=\big\langle w(0),v(t_0)\big\rangle=0\quad\text{\rm and}\quad
\lim_{t\to t_0-}\big\langle w(t),v(t)\big\rangle=\big\langle w(t_0),\varphi\big\rangle.
\end{equation}
This together with \eqref{eq:zero_derivative} then implies that for any $t_0\in(0,T)$ we have that $\big\langle w(t_0),\varphi\big\rangle=0$
for any $\varphi\in\Sz$. This completes the proof of the lemma.
\end{proof}

In the proof of Lemma \ref{le:uniqueness} we use the following 

\begin{Lem}\label{lem:pairing}
Let $-\infty<a<b<\infty$. Then one has:
\begin{itemize}
\item[(i)] If $u\in C\big([a,b],\Sz'\big)$ and $v\in C\big([a,b],\Sz\big)$ then
$[a,b]\to\R$, $t\mapsto\big\langle u(t),v(t)\big\rangle$, is continuous.
\item[(ii)] If $u\in C^1\big((a,b),\Sz'\big)$ and $v\in C^1\big((a,b),\Sz\big)$ then
$(a,b)\to\R$, $t\mapsto\big\langle u(t),v(t)\big\rangle$, is continuously differentiable on $(a,b)$ and 
$\frac{d}{dt}\big\langle u(t),v(t)\big\rangle=\big\langle u'(t),v(t)\big\rangle+\big\langle u(t),v'(t)\big\rangle$.
\end{itemize}
\end{Lem}

\begin{proof}[Proof of Lemma \ref{lem:pairing}]
The lemma follows from the uniform boundedness principle in $\Sz'$. In fact, let us prove item (i):
Take $t_*\in[a,b]$ and assume that $u,v\in C\big((a,b),\Sz'\big)$ and $v\in C\big((a,b),\Sz\big)$.
Take a sequence $(t_k)_{k\ge 1}$ in $[a,b]$ such that $t_k\to t_*$ as $k\to\infty$. Since $u\in C\big((a,b),\Sz'\big)$
we have that $\langle u(t_k),\varphi\rangle\to\langle u(t_*),\varphi\rangle$ as $k\to\infty$ for any given $\varphi\in\Sz$.
This implies that the family of distributions $\big\{u(t_k)\,\big|\,k\ge 1\big\}$ in $\Sz'$ is weakly bounded. 
Then, the uniform boundedness principle implies that there exist integers $m,l\ge 0$ and a constant $C>0$ such that 
for any $k\ge 1$ and for any $\varphi\in\Sz$,
\begin{equation}\label{eq:uniform_boundedness}
\big\langle u(t_k),\varphi\big\rangle\le C\,p_{l,m}(\varphi),\quad 
p_{l,m}(\varphi):=\sup_{x\in\R^d, |\alpha|\le m}\big|\x^l(\partial^\alpha\varphi)(x)\big|.
\end{equation}
For any given $k\ge 1$ we have 
\[
\big|\big\langle u(t_k),v(t_k)\big\rangle-\big\langle u(t_k),v(t_*)\big\rangle\big|\le
\big|\big\langle u(t_k),v(t_k)-v(t_*)\big\rangle\big|
\le C\,p_{l,m}\big(v(t_k)-v(t_*)\big).
\]
This together with the fact that $v(t_k)\stackrel{\Sz'}{\to}v(t_*)$ and $u(t_k)\stackrel{\Sz'}{\to}u(t_*)$ as $k\to\infty$
implies that 
\[
\big\langle u(t_k),v(t_k)\big\rangle\to\big\langle u(t_*),v(t_*)\big\rangle\quad\text{\rm as}\quad k\to\infty
\]
This proves (i). The proof of (ii) follows in a similar way.
\end{proof}

\medskip

For any $t\ge 0$ and $u_0\in\Sz'$ define the transformation $S_\nu(t) : \Sz'\to\Sz'$ by
\begin{equation}\label{eq:S_nu_in_S'}
\big\langle S_\nu(t)u_0,\varphi\big\rangle=\big\langle u_0,G_t*\varphi\big\rangle,\quad\varphi\in\Sz.
\end{equation}
The map $S_\nu(t) : \Sz'\to\Sz'$ extends the heat flow $S_\nu(t) : \Sz\to\Sz$
given by the {\em Poisson integral} $S_\nu(t)\varphi:=G_t*\varphi$, $\varphi\in\Sz$ (see \eqref{eq:Gaussian}).
We have

\begin{Lem}\label{lem:heat_flow_in_S'}
For any given $u\in\Sz'$ one has the following:
\begin{itemize}
\item[(i)] The map 
\begin{equation}\label{eq:heat_flow_in_S'}
[0,\infty)\to\Sz',\quad t\mapsto S_\nu(t)u
\end{equation}
lies in $C\big([0,\infty),\Sz'\big)\cap C^1\big((0,\infty),\Sz'\big)$.
\item[(ii)] The curve $u(t):=S_\nu(t)u_0$, $t\ge 0$, satisfies the heat equation
\begin{equation}\label{eq:heat_equation_appendix}
u'(t)=\nu\Delta u(t),\quad u(0)=u_0,
\end{equation}
in $\Sz'$.
\item[(iii)] For any $t_1, t_2\ge 0$ one has that $S_\nu(t_1+t_2)=S_\nu(t_1) S_\nu(t_2)$.
\end{itemize}
\end{Lem}

\begin{proof}[Proof of Lemma \ref{lem:heat_flow_in_S'}]
The lemma easily follows from the formula
\begin{equation}\label{eq:heat_flow_in_fourier_coordinates}
\widehat{(S_\nu(t)u_0)}(\xi)=e^{-|\xi|^2t}\widehat{u}_0(\xi)
\end{equation}
where $\widehat{(\cdot)}$ denotes the the Fourier transform in $\Sz'$ -- see e.g. \cite{Shubin}, Ch. 7.
\end{proof}

\begin{Rem}\label{rem:weak_analyticity_in_S'}
In fact, it follows from \eqref{eq:heat_flow_in_fourier_coordinates} that for any $\varepsilon>0$ 
the map \eqref{eq:heat_flow_in_S'} extends to an analytic map for $t$ in the half plane 
$\big\{t\in\C\,\big|\,\re(t)>\varepsilon\big\}$.
\end{Rem} 

By combining Lemma \ref{lem:heat_flow_in_S'} with Lemma \ref{le:uniqueness} we obtain
the following

\begin{Prop}\label{prop:heat_equation_in_S'}
For any initial data $u_0\in\Sz'$ the heat equation \eqref{eq:heat_equation_appendix} has a unique solution in the class
$C\big([0,\infty),\Sz'\big)\cap C^1\big((0,\infty),\Sz'\big)$
which is given by $u(t)=S_\nu(t)u_0$, $t\ge 0$, and $S_\nu(t)$ is defined by \eqref{eq:S_nu_in_S'}.
\end{Prop}

\begin{Coro}\label{coro:exp=S}
Let $X$ be a Banach space of functions on $\R^n$ that is continuously embedded into $\Sz'$, $X\subseteq \Sz'$, and let 
$\{e^{t \nu A}\}_{t\ge 0}$ be a strongly continuous semigroup in $X$ with generator $A : D(A)\to X$ so that $A$ coincides with 
the weak Laplacian $\Delta : \Sz'\to\Sz'$ restricted to $D(A)$. Then $e^{t\nu A}=S_\nu(t)|_{X}$ for any $t\ge 0$.
\end{Coro}

\begin{proof}[Proof of Corollary \ref{coro:exp=S}]
Since $\{e^{t \nu A}\}_{t\ge 0}$ is a strongly continuous semigroup with generator $A : D(A)\to X$ we conclude that for any
$u_0\in D(A)$, the curve $u(t):=e^{t\nu A}u_0$ belongs to the space
\begin{equation}\label{eq:D(A)-evolution}
u\in C\big([0,\infty),D(A)\big)\cap C^1\big([0,\infty),X\big)
\end{equation}
where $D(A)$ is a dense subspace of $X$ which is complete when equipped with the graph-norm $\|x\|_D:=\|x\|_X+\|Ax\|_X$, $x\in D(A)$.
Moreover, $u(t)$, $t\ge 0$, satisfies the heat equation \eqref{eq:heat_equation_appendix} since 
\[
u'(t)=\nu A u(t)=\nu\Delta u(t),\quad u(0)=u_0,
\]
by the assumption that $A=\Delta|_{D(A)}$. On the other side, since $X$ is continuously embedded into $\Sz'$,
\eqref{eq:D(A)-evolution} implies that $u\in C\big([0,\infty),\Sz'\big)\cap C^1\big([0,\infty),\Sz'\big)$.
Then, by Proposition \ref{prop:heat_equation_in_S'}, we conclude that $e^{t\nu A}u_0=S_\nu(t)u_0$ for any $t\ge 0$.
Since this holds for any $u_0\in D(A)$ and since $D(A)$ is dense in $X$ we conclude that $e^{t\nu A}=S_\nu(t)|_{X}$ for any $t\ge 0$.
\end{proof}

In Section \ref{sec:W-spaces} we use the following lemma.

\begin{Lem}\label{lem:claderon-zygmund}
Assume that $1<p<\infty$, $d\ge 3$, and let $A : \s^{d-1}\to\R$ be a Lipschitz continuous function of mean value zero
on the unit sphere $\s^{d-1}$ and let $\eta\in C_c^\infty(\R)$ be such that $\eta\equiv 1$ in an open neighborhood of zero in $\R$.
Then the convolution operator $\mathcal{T} : \varphi\mapsto T*\varphi$, $\varphi\in C_c^\infty$, where
\[
T(x):=\mathop{\rm P.V.}\,\frac{A(x/|x|)}{|x|^d}\,\eta(|x|)
\]
and $\mathop{\rm P.V.}$ denotes Cauchy's principal value,
extends to a bounded linear map $\mathcal{T} : L^p\to L^p$.
\end{Lem}

\begin{proof}[Proof of Lemma \ref{lem:claderon-zygmund}]
The lemma follows easily from Theorem 2 in \cite[Ch.\,II,\,\S\,3]{Stein}. 
We include the proof for completeness of the arguments. Let us set
\begin{equation}\label{eq:K}
\Omega(x):=A(x/|x|)\eta(|x|)\quad\text{\rm and}\quad K(x):=\Omega(x)/|x|^d,\quad x\in\R^d.
\end{equation}
The function $\Omega : \R^d\to\R$ has mean value zero on any sphere centered at zero in $\R^d$ and 
there exists a constant $B>0$ such that
\[
|K(x)|\le B/|x|^d
\]
for all $x\ne 0$. Hence, by Theorem 2 in \cite[Ch.\,II,\,\S\,3]{Stein}, the lemma will follow if we show that the integral
\begin{equation}\label{eq:Dini}
\int_{2|y|\le|x|}\big|K(x-y)-K(x)\big|\,dx
\end{equation}
is bounded uniformly in $y\ne 0$. Following \cite[Ch.\,II,\,\S4]{Stein} we write for $x\notin\{0, y\}$,
\begin{equation}\label{eq:stein1}
\big|K(x-y)-K(x)\big|\le\frac{\big|\Omega(x-y)-\Omega(x)\big|}{|x-y|^d}+
|\Omega(x)|\left|\frac{1}{|x-y|^d}-\frac{1}{|x|^d}\right|\,.
\end{equation}
Let us first consider the second term on the right hand side of \eqref{eq:stein1}.
For $0<2|y|\le|x|$ we have
\begin{align*}
\left|\frac{1}{|x-y|^d}-\frac{1}{|x|^d}\right|\le\big||x|-|x-y|\big|\sum_{k+l\le d-1;k,l\ge 0}\frac{1}{|x|^{d-k}|x-y|^{d-l}}
\le C_1\frac{\big||x|-|x-y|\big|}{|x|^{d+1}}
\end{align*}
where $C_1>0$ is a constant depending on $d\ge 3$ and where we used that $|x-y|\ge|x|/2$ for $2|y|\le|x|$.
This implies that there exist constants $C_2, C_3>0$ such that for $y\ne 0$,
\begin{align*}
\int_{2|y|\le|x|}|\Omega(x)|\left|\frac{1}{|x-y|^d}-\frac{1}{|x|^d}\right|\,dx&
\le C_2\int_{2|y|\le|x|}\frac{\big||x|-|x-y|\big|}{|x|^{d+1}}\,dx\\
&\le C_2\int_{|x'|\ge 2}\frac{\big||x'|-|x'-y'|\big|}{|x'|^{d+1}}\,dx'\le C_3,
\end{align*}
where we pass to the variable $x':=x/|y|$ in the integral, set $y':=y/|y|\in \s^{d-1}$, and use
that $\big||x'-y'|-|x'|\big|\le |y'|=1$.

Let us now consider the first term on the right hand side of \eqref{eq:stein1}.
Since $\eta\equiv 1$ in an open neighborhood of zero, there exists $\rho_1>0$ such that
$\Omega(x-y)=\Omega(x)=A(x/|x|)$ for $|x|\le\rho_1$ and $2|y|\le|x|$.
Similarly, since $\eta$ has compact support, there exists $\rho_2>0$ such that $\Omega(x-y)=\Omega(x)=0$ for
$|x|\ge\rho_2$ and $2|y|\le|x|$.
Now, take $y\ne 0$. If $0<|y|\le\rho_1/2$ then
\begin{align}\label{eq:stein2}
\int_{2|y|\le|x|}\frac{\big|\Omega(x-y)-\Omega(x)\big|}{|x-y|^d}\,dx&=
\int_{2|y|\le|x|\le\rho_1}\frac{\big|A(x-y/|x-y|)-A(x/|x|)\big|}{|x-y|^d}\,dx\nonumber\\
&+\int_{\rho_1\le|x|\le\rho_2}\frac{\big|\Omega(x-y)-\Omega(x)\big|}{|x-y|^d}\,dx\,.
\end{align}
The uniform boundedness for $0<|y|\le\rho_1/2$ of the first term on the right hand side of \eqref{eq:stein2} 
is proved in \cite{Stein} -- see the proof of Theorem 3 in \cite[Ch.\,II,\,\S\,4]{Stein}, where one has to use
that $A : \s^{d-1}\to\R$ is Lipschitz continuous. The uniform boundedness for $|y|\le\rho_1/2$ of the second term 
follows since $|x-y|\ge\rho_1/2>0$ in the domain of integration $\rho_1\le|x|\le\rho_2$ (and hence the integrand
is a continuous function of its arguments) and since the set $\{y\in\R^d\,|\,|y|\le\rho_1/2\}$ is compact in $\R^d$. 
Now, assume that $\rho_1/2\le|y|\le\rho_2/2$. Then
\begin{equation}\label{eq:stein3}
\int_{2|y|\le|x|}\frac{\big|\Omega(x-y)-\Omega(x)\big|}{|x-y|^d}\,dx=
\int_{2|y|\le|x|\le\rho_2}\frac{\big|\Omega(x-y)-\Omega(x)\big|}{|x-y|^d}\,dx
\end{equation}
The integral on the right hand side is uniformly bounded for $\rho_1/2\le|y|\le\rho_2/2$ since
$|x-y|\ge|x|/2\ge\rho_1/2>0$ for $2|y|\le|x|$ and since the set 
$\{y\in\R^d\,|\,\rho_1/2\le|y|\le\rho_2/2\}$ is compact in $\R^d$.
The case when  $|y|\ge\rho_2/2$ is trivial since then $|x|\ge\rho_2$, and hence the integral on 
the left hand side of \eqref{eq:stein3} vanishes. This completes the proof of the lemma.
\end{proof}

The following lemma is used in Section \ref{sec:navier-stokes}.

\begin{Lem}\label{lem:T-map_real}
Assume that $m>2+\frac{d}{p}$, $N\in\Z_{\ge 0}$, and  $1<p<\infty$.
For $u\in C\big([0,T),\A^{m,p}_{N;0}\big)$, $u(0)=u_0$, consider the transformation
\begin{equation}\label{eq:T-map_real}
\mathcal{T}_{\rm real}(u)\equiv S_\nu(t)u_0+\int_0^t S_\nu(t-s) F\big(u(s)\big)\,ds,\quad t\in[0,T).
\end{equation}
Then 
$\mathcal{T}_{\rm real}(u)\in
C\big([0,T),\A^{m,p}_{N;0}\big)\cap C^1\big((0,T),\A^{m-2,p}_{N;0}\big)$, $\mathcal{T}_{\rm real}(u)(0)=u_0$, 
and it satisfies the relation $\mathcal{T}_{\rm real}(u)_t=\nu\Delta\mathcal{T}_{\rm real}(u)+F(u)$.
\end{Lem}

\begin{proof}[Proof of Lemma \ref{lem:T-map_real}]
For $k\ge 1$ we set $\alpha_k:=1-\frac{1}{2k}$ and define
\begin{equation}\label{eq:w_k-real}
w_k(t):=\int_0^{\alpha_k t} S_\nu(t-s) F\big(u(s)\big)\,ds\quad\forall t\in(0,T)\,.
\end{equation}
Consider the open set $W_k:=\big\{(t,s)\,\big|\,0<t<T,\,0<s<\alpha_{k+1} t\big\}$ in $\R^2$ and note that
$(t,\alpha_k t)\in W_k$ and $W_k\subseteq W_{k+1}$ for $k\ge 1$. 
It follows from Lemma \ref{lem:lipschitz} and Corollary \ref{coro:S(t)-estimatesA*} that 
the integrand in \eqref{eq:w_k-real},
\[
\mathcal{F} : W_k\to\A^{m,p}_{N;0},\quad (t,s)\mapsto S_\nu(t-s) F\big(u(s)\big),
\]
and its partial derivative in the direction of the first variable
\[
(\partial_t\mathcal{F}_k) : W_k\to\A^{m-2,p}_{N;0},\quad (t,s)\mapsto\nu\Delta\Big(S_\nu(t-s) F\big(u(s)\big)\Big),
\]
are continuous maps. This implies that for any $k\ge 1$,
\[
w_k\in C\big([0,T),\A^{m,p}_{N;0}\big)\cap C^1\big((0,T),\A^{m-2,p}_{N;0}\big)
\]
and
\begin{equation}\label{eq:w_k'-real}
\begin{array}{ccc}
w_k'(t)&=&\alpha_k S_\nu(t/2k) F\big(u(\alpha_k t))+\int_0^{\alpha_k t}\nu\Delta\Big(S_\nu(t-s) F\big(u(s)\big)\Big)\,ds\\
&=&\alpha_k S_\nu(t/2k) F\big(u(\alpha_k t))+\nu\Delta\Big(\int_0^{\alpha_k t}S_\nu(t-s) F\big(u(s)\big)\,ds\Big)
\end{array}
\end{equation}
where the first integral in \eqref{eq:w_k'-real} exists as a Riemann integral in $\A^{m-2,p}_{N;0}$ and
the second integral exists as a Riemann integral in $\A^{m,p}_{N;0}$.
Now, one can use Theorem \ref{th:SonA}, Corollary \ref{coro:S(t)-estimatesA}, Corollary \ref{coro:S(t)-estimatesA*},
Lemma \ref{lem:lipschitz}, to conclude from \eqref{eq:w_k-real} and \eqref{eq:w_k'-real} that
\[
\begin{tikzcd}
w_k\arrow[r, "\A^{m,p}_{N;0}"]&w
\end{tikzcd}
\quad\text{\rm and}\quad 
\begin{tikzcd}
w_k'\arrow[r, "\A^{m-2,p}_{N;0}"]&g
\end{tikzcd}
\quad \text{\rm as}\quad k\to\infty,
\]
where the convergence is uniform on $(0,T')$ for any $0<T'<T$,
\[
w(t)\equiv\int_0^t S_\nu(t-s) F\big(u(s)\big)\,ds,
\]
and $g(t)\equiv F\big(u(t))+\nu\Delta\Big(\int_0^t S_\nu(t-s) F\big(u(s)\big)\,ds\Big)$ for $t\in(0,T)$.
This implies that $w\in C\big((0,T),\A^{m,p}_{N;0}\big)\cap C^1\big((0,T),\A^{m-2,p}_{N;0}\big)$ and 
\begin{equation}\label{eq:w-relation}
w_t=F\big(u(t))+\nu\Delta\Big(\int_0^t S_\nu(t-s) F\big(u(s)\big)\,ds\Big),\quad t\in(0,T)\,.
\end{equation}
If we set $w(0)=0$ then $w : [0,T)\to\A^{m,p}_{N;0}$ is continuous at $t=0$ by Corollary \ref{coro:S(t)-estimatesA}.
Finally, the relation $\mathcal{T}_{\rm real}(u)_t=\nu\Delta\mathcal{T}_{\rm real}(u)+F(u)$ on $(0,T)$ follows 
directly from \eqref{eq:w-relation} and the fact that $\big(S_\nu(t) u_0\big)_t=\nu\Delta\big(S_\nu(t) u_0\big)$ 
by Theorem \ref{th:SonA}.
\end{proof}

One has the following useful characterization of the domain $\widetilde{W}^{m+2,p}_\delta$ of the Laplace operator 
appearing in Theorem \ref{th:sectorial_W}.

\begin{Lem}\label{lem:domains}
Assume that $m\in\Z_{\ge 0}$, $\delta\in\R$, $0\le n\le N$, and $1<p<\infty$.
Then, $\widetilde{W}^{m+2,p}_\delta=W^{m,p}_\delta\cap W^{m+1,p}_{\delta-1}\cap W^{m+2,p}_{\delta-2}$.
\end{Lem}

\noindent The proof of this lemma follows directly from the definition of the involved spaces and will be thus omitted.


\end{document}